\documentclass[10pt]{article}

\usepackage[authoryear,round]{natbib}                   % need for \citet*
\usepackage{amsmath}                                                        % need for subequations
\usepackage{graphicx}                                                       % need for figures
\usepackage{subfigure}
\usepackage{enumitem}                                                    % need for subfigures
\usepackage{hyperref}
\usepackage{xcolor}
\usepackage{amssymb}                                                        % gives you \mathbb{} font
\usepackage[mathscr]{eucal}                                             % gives you \mathscr font
\usepackage{cancel}                                                             % gives you the ability to visibly cross out terms in equations}
\usepackage[normalem]{ulem}                                                                 % gives you \sout
\usepackage{pstricks}
\usepackage{rotating}
\usepackage{lscape}
\usepackage[paperwidth=8.5in,paperheight=11in,top=1.25in, bottom=1.25in, left=1.00in, right=1.00in]{geometry}
\usepackage{mathtools}                                                      % need for `show only references'
\mathtoolsset{showonlyrefs=true}                                    % only equations which are labeled AND referenced will be numbered.
\usepackage{fixltx2e,amsmath}                                           % Supposedly, this allows one to use \eqref{} in \caption{}.
\MakeRobust{\eqref}

\usepackage{mathdots}
\usepackage{amsthm}                                                             % need for theorem-proof environment
\allowdisplaybreaks                                                             % allows page breaks for long equations
\theoremstyle{plain}
\newtheorem{theorem}{Theorem}

\newtheorem{lemma}[theorem]{Lemma}                              % [theorem] ==> theorems and lemmas will share a counter
\newtheorem{proposition}[theorem]{Proposition}
\newtheorem{corollary}[theorem]{Corollary}
\theoremstyle{definition}
\newtheorem{definition}[theorem]{Definition}
\newtheorem{example}[theorem]{Example}
\newtheorem{remark}[theorem]{Remark}
\newtheorem{assumption}[theorem]{Assumption}

\def \s {{\sigma}}
\def \g {{\gamma}}
\def \a {{\alpha}}
\def \b {{\beta}}
\def \d {{\delta}}

\newcommand{\<}{\langle}
\renewcommand{\>}{\rangle}

\newcommand\N{\mathbb{N}}

\newcommand\Nb{\mathbb{N}}
\newcommand\R{\mathbb{R}}

\newcommand\p{\partial}

\newcommand\Gc{\mathscr{G}}

\newcommand\Lc{\mathscr{L}}

\newcommand\Om{\Omega}

\newcommand\xb{\bar{x}}

\def \p {{\partial}}

\def \R  {{\mathbb {R}}}
\def \x {{\xi}}
\def \g {{\gamma}}

\def \z {{\zeta}}
\def \p {{\partial}}
\def \a {{\alpha}}
\def \O {{\Omega}}

\def \d {{\delta}}

\def \a {{\alpha}}
\def \b {{\beta}}
\def \d {{\delta}}

\def \s {{\sigma}}

\def \R {{\mathbb {R}}}
\def \N {{\mathbb {N}}}

\def \x {{\xi}}

\def \z {{\zeta}}

\def \g {{\gamma}}
\def \O {{\Omega}}
\def \phi {{\varphi}}

\def \tilde {\widetilde}

\def\l {\lambda}

\def \à {{\`a }}
\def \è {{\`e }}
\def \ò {{\`o }}
\def \ù {{\`u }}

\newcommand{\norm}[1]{\left\|{#1}\right\|}

\newcommand\dd{\mathrm{d}}

\newcommand\Rd{\mathbb{R}^{d}}
\newcommand\Rdd{\mathbb{R}\times\mathbb{R}^{d}}

\newcommand\Ndzero{\mathbb{N}^d_{0}}
\newcommand\Nzero{\mathbb{N}_{0}}

\newcommand\formaldeg{m}%{\mathfrak{d}}

\def \phi {{\varphi}}

\begin{document}

\title{Intrinsic Taylor formula for Kolmogorov-type homogeneous groups
}
\author{
Stefano Pagliarani
\thanks{Centre de Math\'ematiques Appliqu\'ees, Ecole Polytechnique, Paris, France. The author's research was supported by the Chair {\it Financial Risks} of the {\it Risk Foundation} and the {\it  Finance for Energy Market Research Centre}.}
\and
Andrea Pascucci
\thanks{Dipartimento di Matematica,
Universit\`a di Bologna, Bologna, Italy}
\and
Michele Pignotti
\thanks{Dipartimento di Matematica,
Universit\`a di Bologna, Bologna, Italy}
}

\date{This version: \today}

\maketitle

\begin{abstract}
%%We consider the non-Euclidean quasi-metric on $\R\times \R^d$
%%homogeneous Lie group $()$ induced by the
%%$$
%%\norm{\z^{-1}\circ z}_B,\qquad z,\z \in \R\times \R^d
%%$$
%%induced by
%We consider the class of ultra-parabolic Kolmogorov-type operators% of the kind
%$$
%  \Lc=\frac{1}{2}\Delta_{p_0}+\langle B x, \nabla_x \rangle+\p_{t},\qquad
%  x\in\Rd %(t,x)\in\R\times\R^{d}
%  , \qquad p_0\leq d.
%$$
%satisfying the H\"ormander's condition. After defining suitable functional \emph{intrinsic} H\"older spaces $C^{n,\alpha}_{B}$ and $C^{n,\alpha}_{B,\text{loc}}$ in terms of the regularity %of the function
%along the fields $\partial_{x_1},\cdots, \partial_{x_{p_0}}$ and $Y=\langle B x, \nabla_x
%\rangle{\blue +\p_{t}}$, we prove an \emph{intrinsic} Taylor formula of the kind
%\begin{equation}
%u(z)=T_n u(\z,z)+R_n(z,\z),\qquad z,\z\in\R\times\R^d,
%\end{equation}
%with global and local bounds for the remainder of the type
%\begin{align}
%& |R_n(z,\z)|\leq c_{B} \|u\|_{C^{n,\a}_{B}} \|\z^{-1}\circ z\|_{B}^{n+\a}, \qquad z,\z \in \Rdd&& \text{(if $u\in C^{n,\alpha}_{B}$)},\\
%& |R_n(z,\z)|\leq c_{B,\z} \|u\|_{C^{n,\a}_{B,\text{loc}}} \|\z^{-1}\circ z\|_{B}^{n+\a}, \qquad % z,\z \in \Rdd, \
%\norm{\z^{-1}\circ z}_B<<1 && \text{(if $u\in C^{n,\alpha}_{B,\text{loc}}$)},
%\end{align}
%where $\circ$ and $\norm{\cdot}_B$ are respectively the translation and the norm related to the
%homogeneous Lie structure induced by the matrix $B$.

We consider a class of ultra-parabolic Kolmogorov-type operators satisfying the H\"ormander's
condition. We prove an intrinsic Taylor formula with global and local bounds for the remainder
given in terms of the norm in the homogeneous Lie group naturally associated to the differential
operator.

\end{abstract}


%%%%%%%%%%%%%%%%%%%%%%%%%%%%%%%%%%%%%%%
%
%       SECTION: Introduction
%
%%%%%%%%%%%%%%%%%%%%%%%%%%%%%%%%%%%%%%%

\section{Introduction}\label{intro}
We consider a class of Kolmogorov 
operators of the form
\begin{equation}\label{e1}
  \Lc=\frac{1}{2}\sum_{i=1}^{p_{0}}\p_{x_{i}x_{i}}+\sum_{i,j=1}^{d}b_{ij}x_{j}\p_{x_{i}}+\p_{t},\qquad
  (t,x)\in\R\times\R^{d},
\end{equation}
where $1\le p_{0}\le d$ and $B=\left(b_{ij}\right)$ is a constant $d\times d$ matrix. If $p_{0}=d$
then $\Lc$ is a parabolic operator while in general, for $p_{0}<d$, $\Lc$ is degenerate and not
uniformly parabolic. Some structural assumptions on 
$B$ implying that $\Lc$ is a hypoelliptic operator will be introduced and discussed below.

Operators of the form \eqref{e1} appear in several applications in physics, biology and
mathematical finance. We recall that $\Lc$ is the linearized prototype of the Fokker-Planck
operator arising in fluidodynamics (cf. \cite{Chandresekhar}). Moreover $\Lc$ was extensively
studied by \cite{Kolmogorov91} as the infinitesimal generator of the linear stochastic equation in
$\R^{d}$
\begin{equation}\label{e2}
  \dd X_{t}=B X_{t}\dd t+\s \dd W_{t},\qquad 
\end{equation}
where $W$ is a $p_{0}$-dimensional standard Brownian motion and $\s$ is a $d\times p_{0}$ matrix
such that
  $$\s\s^{T}=\begin{pmatrix}
    I_{p_{0}} & 0 \\
    0 & 0 \
  \end{pmatrix},$$
with $I_{p_{0}}$ being the $p_{0}\times p_{0}$ identity matrix. A particular case of \eqref{e2} is
the well-known Langevin equation from kinetic theory, which in simplified form reads
  $$
  \begin{cases}
    \dd X^{1}_{t}=\dd W_{t} \\
    \dd X^{2}_{t}=X_{t}^{1}\dd t,
  \end{cases}
  $$
where $W$ is a real Brownian motion, and whose generator is the Kolmogorov operator
\begin{equation}\label{e5}
  \frac{1}{2}\p_{x_{1}x_{1}}+x_{1}\p_{x_{2}}+\p_{t},\qquad (t,x_{1},x_{2})\in\R^{3}.
\end{equation}
We also refer to \cite{Talay2011} for a recent study of Navier-Stokes equations involving more
general Kolmogorov-type operators.

In mathematical finance, Kolmogorov equations arise in models incorporating some sort of
dependence on the past: typical examples are Asian options (see, for instance, \cite{Ingersoll},
\cite{BarucciPolidoroVespri}, \cite{Pascucci08}, \cite{NyPasPol}) and some volatility models (see,
for instance, \cite{HobsonRogers} and \cite{FoschiPascucci}).

It is natural to place operator $\Lc$ in the framework of H\"ormander's theory; indeed, let us set
\begin{equation}\label{e3}
  X_{j}=\p_{x_{j}},\quad j=1,\dots,p_{0},\quad\text{ and }\quad Y=\langle
  Bx,\nabla\rangle+\p_{t},
\end{equation}
where $\langle\cdot,\cdot\rangle$ and $\nabla=\left(\p_{x_{1}},\dots,\p_{x_{d}}\right)$ denote the
inner product and the gradient in $\R^{d}$ respectively. Then $\Lc$ can be written as a sum of
vector fields:
  $$\Lc=\frac{1}{2}\sum_{j=1}^{p_{0}}X_{j}^{2}+Y.$$
Under the H\"ormander's condition
\begin{equation}\label{e4}
  \text{rank}\left(\text{Lie}(X_{1},\dots,X_{p_{0}},Y)\right)=d+1,
\end{equation}
operator $\Lc$ is hypoelliptic and \cite{Kolmogorov2} and \cite{Hormander} constructed an explicit
fundamental solution of $\Lc u=0$, {which} is the transition density of $X$ in \eqref{e2}. We
remark that $X$ is a Gaussian process and condition \eqref{e4} turns out to be equivalent to the
non-degeneracy of the covariance matrix of $X_{t}$ for any positive $t$ (see, for instance,
\cite{Karatzas} and \cite{Pascuccibook2011}).

Operator $\Lc$ in \eqref{e1} is the prototype of the more general class of {\it Kolmogorov
operators with variable coefficients.} The study of general Kolmogorov operators has been
successfully carried {out} by several authors in the framework of the theory of homogeneous
groups: \cite{Folland75}, \cite{FollandStein1982}, \cite{Varopoulos} and \cite{Bonfiglioli07}
serve as a reference for the analysis of homogeneous groups. We recall that
\cite{LanconelliPolidoro1994} first studied the non-Euclidean intrinsic geometry induced by
Kolmogorov operators and \cite{Polidoro94}, \cite{amrx} proved the existence of a fundamental
solution under optimal regularity assumptions on the coefficients; in particular,
\cite{Polidoro94} generalized and greatly improved the classical results by \cite{Weber},
\cite{Il'in}, \cite{Sonin} and
\cite{Gencev} where unnecessary Euclidean-type regularity was required. 

The intrinsic Lie group structure modeled on the vector fields $X_{1},\dots,X_{p_{0}},Y$ and the
related non-Euclidean functional analysis (H\"older and Sobolev spaces) were studied by several
authors, among others \cite{Ragusa}, \cite{Francesco}, \cite{BraCEMA}, \cite{Manfredini},
\cite{Lunardi1997}, \cite{Kunze}, \cite{NyPasPol10}, \cite{Priola} and \cite{Menozzi}. When
dealing with intrinsic H\"older spaces, Taylor-type formulas (and the related estimates for the
remainder) form one of the cornerstones for the development of the theory. Classical results about
intrinsic Taylor polynomials on homogeneous groups were proved in great generality by
\cite{FollandStein1982}. Recently, \cite{Bonfiglioli2009} 
derived explicit formulas for Taylor polynomials on homogeneous groups and the corresponding remainders by adapting
the classical Taylor formula with integral remainder.

The main result of this paper is a new and more explicit representation of the intrinsic Taylor
polynomials for Kolmogorov-type homogeneous groups. The distinguished features of our formulas are
as follows:
\begin{itemize}
  \item[i)] in \cite{FollandStein1982} and \cite{Bonfiglioli2009}, Taylor polynomials of order {$n$}
  are defined for functions that are differentiable up to order {$n$} in the Euclidean sense;
  the constants in the error estimates for the remainders (that is, the differences between the function and its Taylor
  polynomials) depend on the norms of the function in the Euclidean H\"older spaces. Conversely, in this paper we define {$n$}-th order Taylor
  polynomials for functions that are regular {\it in the intrinsic sense} and the constants
  appearing in the error estimates {\it depend only on the norms of the intrinsic derivatives up to order {$n$}}.
  At the best of our knowledge, a similar result under such intrinsic regularity assumptions only appeared in \cite{Arena},
  but limited to the particular case of the Heisenberg group.
  Moreover, the fact that we assume intrinsic regularity on the function, as opposed to Euclidean one,
  allows us to yield some global error bounds for the remainders when the function belongs to the \emph{intrinsic global} H\"older spaces.
  This represents another key difference with respect to the existing literature, where such bounds are only local.

  \item[ii)] since the vector fields {$X_{1},\dots,X_{p_{0}}$ do not commute with $Y$},
  there are different representations for the Taylor polynomials 
  depending on the order of the
  derivatives: specifically, the representation in \cite{FollandStein1982} and
  \cite{Bonfiglioli2009} is given as a sum over all possible permutations of the derivatives. {Thus}, computing explicitly the {$n$-th order}
  Taylor polynomials can be very lengthy since the number of
  terms {involved} %in the Taylor polynomial of order $n$
  grows proportionally to $d^{n}$.
  On the contrary, even though our Taylor polynomials are %indeed
  algebraically equivalent to those given by \cite{FollandStein1982} and
  \cite{Bonfiglioli2009}, in Theorem \ref{th:main} %(cf. \eqref{eq:def_Tayolor_n})
  we determine a privileged way to order the vector fields so that we are able to get compact Taylor polynomials {with a number of terms increasing
  linearly with respect to the order of the polynomial itself} (see \eqref{eq:def_Tayolor_n} %\eqref{eq:estim_tay_n_loc} and \eqref{eq:estim_tay_n}
  below); 
  {this is quite relevant for practical computations}, as we
  will show through a simple example in {Section \ref{sec:comparison_taylor}}.

\item[iii)] besides the theoretical interest, our result might be useful for diverse applications.
For instance, in the recent works by \cite{LPP4} and \cite{PP_compte_rendu} the authors have
developed a perturbative technique to analytically approximate the solution of a parabolic Cauchy
{problem} with variable coefficients. The Taylor polynomials of the coefficients and of the
terminal datum play an important role in this technique. For instance, the short-time precision of
the approximation turns out to be dependent on the regularity of the terminal datum. Within this
prospective, an intrinsic Taylor formula represents a crucial ingredient in order to extend such
results to the case of ultra-parabolic (i.e. $p_0<d$) Kolmogorov operators with variable
coefficients. In particular, the intrinsic regularity of the coefficients and of the terminal
datum can be exploited to improve the accuracy of the approximate solutions. We refer to Section
\ref{sec:asympt_exp} for further details.
\end{itemize}
The paper is organized as follows: in the next section we state the structural hypothesis on the
matrix $B$, we give the definition of \emph{intrinsic H\"older spaces} and we state our main
result. In Section \ref{sec:3} we %present an overview of comparisons with
{review and compare with} the previous literature (Sections \ref{sec:comparison_taylor} and
{\ref{sec:comparison_holder}}) and present examples and applications (Section
\ref{sec:asympt_exp}). In Section \ref{sec:prelim} we prove some results that are {preliminary to
the proof of }the main theorem, which will be eventually proved in Section \ref{sec:proof}.

\section{H\"older spaces and Taylor expansions}\label{subseq:definitions}
As first observed by \cite{LanconelliPolidoro1994}, operator $\Lc$ in \eqref{e1} has the
remarkable property of being invariant with respect to left translations in the group
$\left(\Rdd,\circ\right)$, where the non-commutative group law ``$\circ$'' is defined by
\begin{equation}\label{eq:translation}
 (t,x)\circ (s,\xi) = \left(t+s,e^{t B}x+\xi\right),\qquad (t,x),(s,\xi)\in \Rdd.
\end{equation}
Precisely, we have
\begin{equation}\label{eq:translation_invariance}
\big(\Lc u^{(s,\xi)}\big) (t,x)=(\Lc u)\big((s,\xi)\circ(t,x)\big), \qquad (t,x),(s,\xi)\in\Rdd,
\end{equation}
where%, for any $(s,\xi)\in\Rdd$,
\begin{equation}
 u^{(s,\xi)} (t,x)=u((s,\xi)\circ (t,x)).
\end{equation}
Notice that $(\Rdd,\circ)$ is a group with the identity element ${\text{Id}=(0,0)}$ and inverse $(t,x)^{-1}=\left(-t,e^{-tB}x\right)$.%, $(t,x)\in\Rdd$

\cite{LanconelliPolidoro1994} proved that the H\"ormander's condition \eqref{e4} is equivalent to
the following one: for a certain {basis} on $\R^{d}$, the matrix $B$ takes the form
\begin{equation}\label{eq:B_blocks}
B=\left(
\begin{array}{ccccc}
\ast&\ast&\cdots&\ast&\ast\\ B_1 & \ast &\cdots& \ast & \ast \\ 0 & B_2 &\cdots& \ast& \ast \\
\vdots & \vdots &\ddots& \vdots&\vdots \\ 0 & 0 &\cdots& B_r& \ast
\end{array}
\right)
\end{equation}
where each $B_j$ is a $p_j\times p_{j-1}$ matrix of rank $p_j$ with
\begin{equation}
p_0\geq p_1\geq \cdots \geq p_r\geq 1, \qquad \sum_{j=0}^r p_j = d,
\end{equation}
and the $\ast$-blocks are arbitrary. Moreover, if (and only if) the $\ast$-blocks in
\eqref{eq:B_blocks} are null then {\it $\Lc$ is homogeneous of degree two} with respect the
dilations $\left(D(\lambda)\right)_{\l>0}$ on $\Rdd$ given by
\begin{equation}\label{eq:dilation}
 D(\lambda)=\textrm{diag}\big(\lambda^2,\lambda I_{p_0},\lambda^{3}I_{p_1},\cdots,\lambda^{2r+1}I_{p_r}\big),%(t,x), \qquad (t,x)\in\Rdd,
\end{equation}
where $I_{p_j}$ are $p_j\times p_j$ identity matrices: specifically, we have
\begin{equation}\label{eq:dilation_invariance}
 \big(\Lc u^{(\lambda)}\big)(t,x)=\lambda^2(\Lc u)\big(D(\lambda)(t,x)\big), \qquad (t,x)\in\Rdd,\ \lambda>0 ,
\end{equation}
where
\begin{equation}
 u^{(\lambda)} (t,x)=u(D(\lambda)(t,x)).
\end{equation}

Throughout this paper we assume the following standing
%\noindent {\bf Hypothesis H.1.} {\it $B$ is a $d\times d$ constant matrix as in
%\eqref{eq:B_blocks}, where each block $B_j$ has rank $p_j$ and each $\ast$-block is null.}
\begin{assumption}\label{assume:horm}
{\it $B$ is a $d\times d$ constant matrix as in \eqref{eq:B_blocks}, where each block $B_j$ has
rank $p_j$ and each $\ast$-block is null.}
\end{assumption}

\begin{remark}
Under {Assumption \ref{assume:horm}}, the matrix $B$ uniquely identifies the \emph{homogeneous Lie
group} (in the sense of \cite{FollandStein1982})
 $$\mathcal{G}_B:=\left(\Rdd,\circ,D(\lambda)\right).$$
\end{remark}
We define the $D(\l)$-homogeneous norm on $\mathcal{G}_B$ as follows:
\begin{equation}\label{e7}
 \norm{(t,x)}_B=|t|^{1/2}+|x|_B,\qquad |x|_B=\sum_{j=1}^d |x_j|^{1/q_j},%{\blue \qquad (t,x)\in\R\times\R^d,}
\end{equation}
where $(q_j)_{1\leq j\leq d}$ are the integers such that
\begin{equation}
 D(\lambda)=\textrm{diag}\big(\lambda^2,\lambda^{q_1},\cdots,\lambda^{q_d} \big).
\end{equation}
For any $\z\in\Rdd$, we denote by
\begin{equation}\label{e8}
  D_{B}(\z,r)=\{z\in\Rdd\mid \norm{\z^{-1}\circ z}_B< r\}
\end{equation}
the open ball of radius $r$, centered at $\z$, in the homogeneous group $\mathcal{G}_B$.

\begin{remark} There exist two constants $C_1\geq 1$ and $C_2>0$, both depending only on $B$, such that
\begin{align}
 \norm{\z\circ z}_B & \leq C_1\left(\norm{\z\circ \eta}_B+\norm{\eta^{-1}\circ z}_B\right),\qquad && z,\z,\eta \in \Rdd,\\
 \frac{1}{C_2}|z-\z| & \leq \norm{\z^{-1}\circ z}_B\leq C_2|z-\z|^{2r+1}, && \text{for } |z-\z|, \norm{\z^{-1}\circ z}_B \leq 1.
\end{align}
The first inequality implies that $\norm{\cdot}_B$ is a quasi-norm, % is a weaker version of the classic triangular inequality %for the distance induced by the homogeneous norm $\norm{\cdot}_B$
while the second formula shows that {the intrinsic distance is locally equivalent to the
Euclidean one.} %the topology induced by the norm $\norm{\cdot}_B$ is the Euclidean one
%although the distances are not (globally) equivalent.
For a proof we refer to \cite{Manfredini}, Proposition 2.1.
\end{remark}

Next we introduce the notions of $B$-intrinsic regularity and $B$-H\"older space. Let  $X$ be a
Lipschitz vector field on $\Rdd$. For any $z\in\Rdd$, we denote by $\d\mapsto e^{\d X }(z)$ the
integral curve of $X$ defined as the unique solution of
\begin{equation}
\begin{cases}
 \frac{d}{d\d}e^{\d X }(z)= X\left(e^{\d X }(z)\right),\qquad  &\delta\in\R, \\
 e^{\d X }(z)\vert_{\d=0}= z.
\end{cases}
\end{equation}
Explicitly, if $X\in\{X_1,\cdots,X_{p_0},Y\}$ is one of the vector fields in \eqref{e3}, we have
\begin{equation}\label{eq:def_curva_integrale_campo}
 e^{\d X_{i} }(t,x)=(t,x+\delta e_i),\quad i=1,\cdots,p_0,\qquad
 e^{\d Y }(t,x)=(t+\delta,e^{\delta B}x),
\end{equation}
for any $(t,x)\in\Rdd$.

In order to define a {\it gradation} for the homogeneous group $\mathcal{G}_B$ (see Section 2.4 in
\cite{Bonfiglioli2009}), we associate a {\it formal degree} $m_{X}\in\R_{+}$ to each
$X\in\{X_1,\cdots,X_{p_0},Y\}$ in the following canonical way:
\begin{assumption}\label{assume:formal_degrees}
{\it The formal degrees of the vector fields $X_{1},\cdots,X_{p_{0}}$ and $Y$ are set as
$\formaldeg_{X_{j}}=1$ for $1\leq j \leq p_0$ and $\formaldeg_{Y}=2$.}
\end{assumption}
Next we recall the general notion of Lie differentiability and H\"older regularity.
\begin{definition}\label{def:intrinsic_alpha_Holder3}
Let $X$ be a Lipschitz vector field and $u$ be a real-valued function defined in a neighborhood of
$z\in \Rdd$. We say that $u$ is \emph{$X$-differentiable} in $z$ if the function $\d\mapsto
u\left(e^{\d X }(z)\right)$ is differentiable in $0$. {We will refer to the function $z \mapsto
\frac{d}{d \d} u\left(e^{\d X }(z)\right)\big|_{\d=0}$ as \emph{$X$-Lie derivative of $u$}, or
simply \emph{Lie derivative of $u$} when the dependence on the field $X$ is clear from the
context.}
\end{definition}
\begin{definition}\label{def:intrinsic_alpha_Holder}
Let $X$ be a Lipschitz vector field on $\Rdd$ with formal degree $\formaldeg_{X}>0$. For
$\a\in\,]0,\formaldeg_{X}]$, we say that $u\in C_{X}^{\alpha}$ if the semi-norm
\begin{equation}
 \norm{u}_{C^{\a}_{X}}:=%\sup_{z\in\Rdd}|f(z)|+
 \sup_{z\in\Rdd\atop \d\in\R\setminus\{0\}} \frac{
 \left|u\left(e^{\delta X }(z)\right)-
 u(z)\right|}{|\delta|^{\frac{\alpha}{\formaldeg_{X}}}}
\end{equation}
is finite.
\end{definition}

\medskip\noindent Now, let $\O$ be a domain in $\Rdd$. For any $z\in\O$ we set
  $$\d_{z}=\sup\left\{\bar{\d}\in\,]0,1]\mid e^{\d X}(z)\in\O\text{ for any }\d\in [-\bar{\d},\bar{\d}]\right\}.$$
If $\O_{0}$ is a bounded domain with $\overline{\O}_{0}\subseteq\O$, we set
  $$\d_{\O_{0}}=\min_{z\in \overline{\O}_{0}}\d_{z}.$$
Note that $\d_{\O_{0}}\in\,]0,1]$.
\begin{definition}\label{def:intrinsic_alpha_Holder2}
For $\a\in\,]0,\formaldeg_{X}]$, we say that $u\in C_{X,\text{\rm loc}}^{\alpha}(\O)$ if for any
bounded domain $\O_{0}$ with $\overline{\O}_{0}\subseteq\O$, the semi-norm
\begin{equation}\label{e12}
 \left\|u\right\|_{C_{X}^{\alpha}(\O_{0})}:=\sup_{z\in \O_{0}\atop 0<|\d|<\d_{\O_{0}}} \frac{\left|u\left(e^{\delta X }(z)\right)-
 u(z)\right|}{|\delta|^{\frac{\alpha}{\formaldeg_{X}}}}
\end{equation}
is finite.
\end{definition}

Now we define the intrinsic H\"older spaces on the homogeneous group $\mathcal{G}_B$.
\begin{definition}\label{def:C_alpha_spaces}
Let $\a\in\,]0,1]$, then:
\begin{itemize}
  \item [i)] $u\in C^{0,\a}_{B}$ if $u\in C^{\a}_{Y}$ and $u\in C^{\a}_{\p_{x_{i}}}$ for any $i=1,\dots,p_{0}$. For any $u\in C^{0,\a}_{B}$ we
  define the semi-norm
\begin{equation}\label{e9}
  \norm{u}_{C^{0,\a}_{B}}:=\norm{u}_{C^{\a}_{Y}}+\sum_{i=1}^{p_0} \norm{u}_{C^{\a}_{\partial_{x_i}}}.
\end{equation}
  \item [ii)] $u\in C^{1,\a}_{B}$ if $u\in C^{1+\a}_{Y}$ and $\p_{x_{i}}u\in C^{0,\a}_{B}$ for any
  $i=1,\dots,p_{0}$. For any $u\in C^{1,\a}_{B}$ we define  the semi-norm
\begin{equation}\label{e10}
  \norm{u}_{C^{1,\a}_{B}}:=\norm{u}_{C^{\a+1}_{Y}}+\sum_{i=1}^{p_0} \norm{\partial_{x_i}u}_{C^{0,\a}_{B}}.
\end{equation}
  \item [iii)] {For} $k\in\Nb$ with $k\ge2$, $u\in C^{k,\a}_{B}$ if $Yu\in C^{k-2,\a}_{B}$ and $\p_{x_{i}}u\in C^{k-1,\a}_{B}$ for any
  $i=1,\dots,p_{0}$. For any $u\in C^{k,\a}_{B}$ we define  the semi-norm
\begin{equation}\label{e11}
  \norm{u}_{C^{k,\a}_{B}}:=\norm{Y u}_{C^{k-2,\a}_{B}}+\sum_{i=1}^{p_0} \norm{\partial_{x_i}u}_{C^{k-1,\a}_{B}}.
\end{equation}
\end{itemize}
Similarly, according to Definition \ref{def:intrinsic_alpha_Holder2}, we define the spaces
$C^{k,\a}_{B,\text{\rm loc}}(\O)$ of {\it locally} H\"older continuous functions on a domain $\O$
of $\Rdd$, and the related semi-norms $\norm{\cdot}_{C^{k,\a}_{B}(\O_{0})}$ on bounded domains
$\O_{0}$ with $\overline{\O}_{0}\subseteq\O$.
\end{definition}
%Next we define the spaces of locally H\"older continuous functions on a domain $\O$ of $\Rdd$.
%\begin{definition}\label{def:C_alpha_spaces_loc}
%Let $\a\in\,]0,1]$. Then
%\begin{itemize}
%  \item [i)]  $u\in C^{0,\a}_{B,\text{\rm loc}}(\O)$ if $u\in C^{\a}_{Y,\text{\rm loc}}(\O)$ and $u\in C^{\a}_{\p_{x_{i}},\text{\rm loc}}(\O)$ for any $i=1,\dots,p_{0}$;
%  \item [ii)] $u\in C^{1,\a}_{B,\text{\rm loc}}(\O)$ if $u\in C^{1+\a}_{Y,\text{\rm loc}}(\O)$ and $\p_{x_{i}}u\in C^{0,\a}_{B,\text{\rm loc}}(\O)$ for any
%  $i=1,\dots,p_{0}$;
%  \item [iii)] for $k\in\Nb$ with $k\ge2$, $u\in C^{k,\a}_{B,\text{\rm loc}}(\O)$ if $Yu\in C^{k-2,\a}_{B,\text{\rm loc}}(\O)$ and $\p_{x_{i}}u\in C^{k-1,\a}_{B,\text{\rm loc}}(\O)$ for any
%  $i=1,\dots,p_{0}$.
%\end{itemize}
%Moreover for $u\in C^{k,\a}_{B,\text{\rm loc}}(\O)$ and $\O_{0}$ bounded domain with
%$\overline{\O}_{0}\subseteq\O$, we define
%  $\norm{u}_{C^{k,\a}_{B}(\O_{0})}$
%by analogy with \eqref{e9}-\eqref{e10}-\eqref{e11}.
%%which is clearly finite for any $\z\in\O$ and $r>0$ such that
%%$\overline{D_{B}(\z,r)}\subseteq\O$.
%%  $$
%%  \norm{u}_{C^{k,\a}_{B}(D_{B}(\z,r))}:=\norm{Y u}_{C^{k-2,\a}_{B}}+\sum_{i=1}^{p_0} \norm{\partial_{x_i}u}_{C^{k-1,\a}_{B}}.
%%  $$
%\end{definition}

\begin{remark}\label{rem:inclusions}
The following inclusion holds{\blue :}
%\begin{equation}
% C^{k+1,\a}_{B,\text{\rm loc}} \subseteq C^{k,\a}_{B,\text{\rm loc}},\qquad k\geq 0.
$C^{k,\a}_{B,\text{\rm loc}} \subseteq C^{k',\a'}_{B,\text{\rm loc}}$ for $0\leq k'\leq k$ and $0<
\a' \leq \a \leq 1$. Moreover we have
%\end{equation}
%Nevertheless, we have
%and
%\begin{equation}
 $C^{k,\a}_{B}\subseteq C^{k,\a}_{B,\text{\rm loc}}$ for $k\geq 0.$
%\end{equation}
\end{remark}

\medskip In the sequel, $\beta=(\beta_1,\cdots, \beta_d)\in \Ndzero$ will denote a multi-index. As
usual
 $$|\beta|:=\sum_{j=1}^d \beta_j\quad \text{ and }\quad \beta!:=\prod_{j=1}^d \left(\beta_j !\right)$$
are called the length and the factorial of $\b$ respectively. Moreover, for any $x\in\Rd$, we set
 $$x^{\beta}=x_1^{\beta_1}\cdots x_d^{\beta_d}\quad \text{ and }\quad
 \partial^{\beta}=\partial^{\beta}_x=\partial_{x_1}^{\beta_1}\cdots
 \partial_{x_d}^{\beta_d}.$$
We also introduce the $B$-length of $\b$ defined as
 $$|\beta|_B:=\sum_{i=0}^r (2i+1) \big| \beta^{[i]} \big|$$
where $\beta^{[i]}\in\Ndzero$ is the multi-index
\begin{equation}\label{eq:sub_multi_index}
  \beta^{[i]}_{k}:=
  \begin{cases}
    \beta_{k} & \text{for}\ \bar{p}_{i-1}< k \le\bar{p}_{i}, \\
    0 & \text{otherwise},
  \end{cases}
\end{equation}
with
\begin{equation}\label{e6}
  \bar{p}_{i}=p_{0}+p_{1}+\cdots+p_{i},\qquad 0\le i\le r,
\end{equation}
and $\bar{p}_{-1}\equiv 0$. We are now in position to state our main result.
\begin{theorem}\label{th:main}
Let $\O$ be a domain of $\Rdd$, $\alpha\in ]0,1]$ and $n\in\Nzero$. If $u\in C^{n,\a}_{B,\text{\rm
loc}}(\O)$ then we have:
\begin{enumerate}
\item[1)] there exist the derivatives
\begin{equation}\label{eq:maintheorem_part1}
 Y^k \partial_x^{\beta}u\in C^{n-2k-|\beta|_B,\alpha}_{B,\text{\rm loc}}(\O),
 \qquad {0\leq 2 k + |\beta|_B \leq n};
\end{equation}
\item[2)]%9
for any $\z\in\O$ there exist $r_{\z},R_{\z}>0$ such that $\overline{D_B(\z,R_{\z})}\subseteq \O$ 
and
\begin{equation}\label{eq:estim_tay_n_loc}
 \left|u(z)-T_n u(\z,z)\right|\le c_{B,\z} \|u\|_{C^{n,\a}_{B{,\text{\rm loc}}}(D_B(\z,R_{\z}))}\|\z^{-1}\circ z\|_{B}^{n+\a},\qquad z\in D_B(\z,r_{\z}),
\end{equation}
where $c_{B,\z}$ is a constant that depends on $B$ and $\z$, while $T_n u(\z,\cdot)$ is the
\emph{$n$-th order $B$-Taylor polynomial of $u$ around $\z=(s,\xi)$} defined as
 \begin{equation}
 T_n u(\z,z):=  \sum_{{0\leq 2 k + |\beta|_B \leq n}}\frac{1}{k!\,\beta!}%\bigg(  \prod_{i=0}^r \frac{1}{\big| \alpha^{[i]} \big|!}  \bigg)
 \big( Y^k \partial_{\xi}^{\beta}u(s,\xi)\big) (t-s)^k\big( x-e^{(t-s)B}\xi  \big)^{\beta},\qquad
 z=(t,x)\in\Rdd; \label{eq:def_Tayolor_n}
\end{equation}
\item[3)] if $u\in C^{n,\a}_{B}$ then we have
\begin{align}\label{eq:ste31}
 Y^k \partial_x^{\beta}u\in C^{n-2k-|\beta|_B,\alpha}_B\qquad \text{ for }\ {0\leq 2 k + |\beta|_B \leq n},
\end{align}
and
\begin{equation}\label{eq:estim_tay_n}
  \big| u(z)-T_n u(\z,z) \big|\le c_B \|u\|_{C^{n,\a}_{B}}  \|\z^{-1}\circ z\|_{B}^{n+\a}, \qquad {z,\z\in\Rdd,}%z=(t,x),\z={\blue (s,\xi)}\in\Rdd,
\end{equation}
where $c_B$ is a positive constant that only depends on $B$.
\end{enumerate}
\end{theorem}
A direct consequence of estimate \eqref{eq:estim_tay_n} %of Theorem \ref{eq:maintheorem_part1}
in the particular case $n=0$ is the following
\begin{corollary}\label{ccaa}
A function $u\in C^{0,\a}_{B}$ if and only if there exists a positive constant $c$ such that
  $$\left|u(z)-u(\z)\right|\le c \norm{\z^{-1}\circ z}_B^{\a},\qquad z,\z\in\Rdd,$$
i.e. $u$ is $B$-H\"older continuous in the sense of Definition 1.2 in \cite{Polidoro94}.
\end{corollary}
For a comparison between intrinsic and Euclidean H\"older continuity we refer to Proposition 2.1
in \cite{Polidoro94}.
\begin{corollary}
If $u\in C^{2r+1,\a}_{B,\text{\rm loc}}(\O)$, then there exists %$\p_t u(z)$ for every $z\in \Rdd $ and
$\p_t u\in C^{0,\a}_{B,\text{\rm loc}}(\O)$. Moreover, we have
\begin{equation}\label{eq:time-derivative-2}
\p_t u(t,x)=Yu(t,x)- \< Bx,\nabla u(t,x )\>.
\end{equation}
\end{corollary}
\proof In Theorem \ref{th:main} we take $\z=(t,x)$, $z=(t+\d,x)$ and note that, in this case, the
spatial increments become
\begin{equation}
 x-e^{\d B}x = -\d Bx + O(\d^2) \qquad \text{ as } \d \to 0.
\end{equation}
Now, by Theorem \ref{th:main} all the spatial first{-order} derivatives exist and
\begin{equation}
u(z)-T_{2r+1}u(\z,z)= u(t+\d,x)-u(t,x)-\d Yu(t,x)+\d \sum_{i=1}^d  \p_{x_i}u(t,x)(Bx)_i +O(\d^2),
\qquad \text{ as } \d \to 0.
\end{equation}
Since
\begin{equation}
\|\z^{-1}\circ z\|_B^{2r+1+\a} = \|(\d, x-e^{\d B}x)\|_B^{2r+1+\a} = O(|\d|^{1+\frac{\a}{2r+1}}),
\qquad \text{ as } \d \to 0,
\end{equation}
we get
\begin{equation}
 \frac{u(t+\d,x)-u(t,x)}{\d}-Yu(t,x)+ \sum\limits_{i=1}^d (Bx)_i \p_{x_i}u(t,x)=O(|\d|^{\frac{\a}{2r+1}}) \qquad \text{as } \d\to
 0.
\end{equation}
This implies that the time-derivative exists and formula \eqref{eq:time-derivative-2} holds. Now,
it also easily follows that $\p_t u\in C^{0,\a}_{B,\text{\rm loc}}(\O)$ since %, by the inclusions in Remark \ref{rem:inclusions},
all the derivatives appearing in the right-hand side of
\eqref{eq:time-derivative-2} are in $ C^{0,\a}_{B,\text{\rm loc}}(\O)$.
\endproof

\section{Comparison with known results, examples and applications}\label{sec:3}
{
\subsection{Taylor formulas for homogeneous Lie groups}\label{sec:comparison_taylor}
Our results can be seen {within} %under
the %far
more general setting of \emph{homogeneous Lie groups} (cf.
\cite{FollandStein1982}). A Lie group $\mathcal{G}=(\R^N,*)$ is said to be homogeneous if there
exists a family of group automorphisms of $\mathcal{G}$, $(D_{\lambda})_{\lambda>0}$, called
dilations,  of the form
\begin{equation}
 D_{\lambda}(x_1,\dots,x_N)=(\lambda^{\sigma_1}x_1,\dots,\lambda^{\sigma_N}x_N),\qquad \lambda>0,
 \end{equation}
%with
{for some $1\leq \sigma_1\leq \sigma_2\leq \cdots\leq \sigma_N$}. The existence of such dilations
implies that the exponential map $\mathrm{Exp}$ between the Lie algebra $\mathfrak{g}$ and
$\mathcal{G}$ is a global diffeomorphism whose inverse is denoted by $\mathrm{Log}$. Moreover, we
have a privileged basis on $\mathfrak{g}$, the Jacobian one, whose elements are the left-invariant
vector fields $Z_i$ uniquely defined by
 $$
 Z_i\vert_{x=0}\equiv \p_{x_i} \qquad i=1,\dots,N.
 $$
In this framework it is natural to define the intrinsic degree of $Z_i$ %to be
as $\sigma_i$ and the $D_{\lambda}$-homogeneous norm
\begin{equation}
|x|_{\mathbb{G}}=\sum_{i=1}^N |x_i|^{\frac{1}{\sigma_i}}.
\end{equation}
Following \cite{Bonfiglioli2009}, the {$n$-th order} intrinsic Taylor polynomial $P_n f(x_0,\cdot)$ of %order $n$ for
a function $f$ {around the} %at the
point $x_0$, can be defined as the unique polynomial function such that
\begin{equation}
\label{eq:def_pol_taylor_bonf}
 f(x)- P_n f(x_0,x) = O(|x_0^{-1}*x|_{\mathbb{G}}^{n+\varepsilon}) \qquad \text{as } |x_0^{-1}*x|_{\mathbb{G}}\to 0,
 \end{equation}
for some $\varepsilon>0$. For $f\in C^{n+1}$ existence and uniqueness of $P_n f$ was proved in
\cite{FollandStein1982}; under the same hypothesis, a more explicit expression and a better
estimate of the remainder was given in \cite{Bonfiglioli2009}. Precisely, in the latter the author
proved that
 \begin{equation}\label{eq:pol_taylor_bonf}
  P_n f(x_0,x) = f(x_0)+ \sum_{k=1}^n \sum_{\genfrac{}{}{0pt}{}{1\leq i_1,\dots,i_k\leq N}{I=(i_1,\dots,i_k),\:\: \sigma(I)\leq n}  }\frac{Z_I f(x_0)}{k!}\mathrm{Log}_{i_1}(x_0^{-1}*x)\cdots \mathrm{Log}_{i_k}(x_0^{-1}*x).
  \end{equation}
Here $\sigma(I):= i_1\sigma_{i_1}+\cdots+i_k\sigma_{i_k} $ denotes the intrinsic order of the
operator $Z_I:=Z_{i_1}\cdots Z_{i_k}$ and $\mathrm{Log}_{i}$ is the $i$-th component of the
$\mathrm{Log}$ map in the basis $\{Z_1,\dots,Z_N\}$.

{Note that, in general, operators $Z_i$ do not commute. Therefore, formula
\eqref{eq:pol_taylor_bonf} typically involves a large number of terms. In the special case of a
Kolmogorov-type group, the Taylor polynomial \eqref{eq:def_Tayolor_n} is much more compact that
\eqref{eq:pol_taylor_bonf} because we can exploit the fact that all but one of the $Z_i$ coincide
with Euclidean derivatives and {thus} commute {with each other}; moreover, {our increments along
the integral curves of the vector fields are different from those}
%we use increments along the integral curves of the vector fields that are different from the increments
in \eqref{eq:pol_taylor_bonf}. We illustrate this fact in the following example.

Let us consider the simplest Kolmogorov group, namely the one induced by the operator defined in
\eqref{e5}. This case corresponds to the matrix }
\begin{equation}\label{eq:B_matrix_kolm_prot}
 B=\begin{pmatrix}
 0 & 0 \\ 1 & 0
 \end{pmatrix},
\end{equation}
and the dilations $D(\lambda) $ take the following explicit form:
\begin{equation}
D(\lambda)(t,x_1,x_2)=(\lambda^2 t,\lambda x_1,\lambda^3 x_2),\qquad (t,x_1,x_2)\in \R\times \R^2.
\end{equation}
Moreover, if $z=(t,x_1,x_2), \z=(s,\x_1,\x_2)$, then we also have
\begin{equation}
\z\circ z =(s+t,x_1+\x_1,x_2+\x_2 + t\x_1),\qquad \z^{-1}\circ z=(t-s,x_1-\x_1,x_2-\x_2 - (t-s)\x_1).
\end{equation}
The components of {left-hand side vector in} the previous formula are exactly the increments
{appearing in \eqref{eq:def_Tayolor_n}}%Theorem \ref{th:main}
. With regard to formula \eqref{eq:pol_taylor_bonf}, %with a slight change of notation,
we have
\begin{equation}
Z_0 = Y,\qquad Z_1 = \p_{x_1},\qquad Z_2 = \p_{x_2},
\end{equation}
while the corresponding components of the $\mathrm{Log}$ map are
\begin{equation}
\mathrm{Log}_0(\z^{-1}\circ z) = t-s,\quad \mathrm{Log}_1(\z^{-1}\circ z)=x_1-\x_1,\quad \mathrm{Log}_2(\z^{-1}\circ z) = x_2-\x_2 - (t-s)\x_1 -\frac{(t-s)(x_1-\x_1)}{2}.
\end{equation}
Note that the first two components coincide with the increments mentioned above while the third
one is different. It follows that, up to order two, the two versions of the Taylor polynomial
coincide. On the other hand, %following
{according to} our definition, the third and fourth polynomials are given by
\begin{align}
 T_3 u(\z,z)&=T_2 u(\z,z) + \frac{1}{3!}\p_{x_1}^3u(\z)(x_1-\x_1)^3 + Y\p_{x_1}u(\z)(x_1-\x_1)(t-s)+ \p_{x_2}u(\z)(x_2-\x_2-(t-s)\x_1),\\
 T_4 u(\z,z)&=T_3 u(\z,z) + \frac{1}{4!}\p_{x_1}^4 u(\z)(x_1-\x_1)^4 + \frac{1}{2!}Y\p_{x_1}^2u(\z)(x_1-\x_1)^2(t-s) \\
            & + \frac{1}{2!}Y^2u(\z)(t-s)^2 + \p_{x_2}\p_{x_1}u(\z)(x_1-\x_1)(x_2-\x_2-(t-s)\x_1),
\intertext{while, according {to} formula \eqref{eq:pol_taylor_bonf}, we have}
 T_3 u(\z,z)&=T_2 u(\z,z) + \frac{1}{2!}
 (Y\p_{x_1}+\p_{x_1}Y)u(\z)(x_1-\x_1)(t-s)+ \frac{1}{3!}\p_{x_1}^3u(\z)(x_1-\x_1)^3\\
  & + \p_{x_2}u(\z)\Big(x_2-\x_2-(t-s)\x_1-\frac{(t-s)(x_1-\x_1)}{2}\Big),\\
 T_4 u(\z,z)&=T_3 u(\z,z) + \frac{1}{2!}Y^2u(\z)(t-s)^2 + \frac{1}{4!}\p_{x_1}^4 u(\z)(x_1-\x_1)^4 \\
 &+ \frac{1}{3!}(Y\p_{x_1}^2+\p_{x_1}Y\p_{x_1}+\p_{x_1}^2 Y)u(\z)(x_1-\x_1)^2(t-s) \\
  &  + \p_{x_2}\p_{x_1}u(\z)(x_1-\x_1)\Big(x_2-\x_2-(t-s)\x_1-\frac{(t-s)(x_1-\x_1)}{2}\Big).
\end{align}
Notice that the above expressions of the Taylor polynomials can be proved to be algebraically
equivalent by using the identity $\p_{x_1}Y=Y\p_{x_1} + \p_{x_2}$.

\subsection{Intrinsic H\"older spaces in the literature}\label{sec:comparison_holder}
{Intrinsic H\"older spaces play a central role in the study of the existence and the regularity
properties of solutions to Kolmogorov operators with variables coefficients. In order to prove
Schauder-type estimates, different notions of H\"older spaces have been proposed by several
authors (see, for instance, \cite{Manfredini}, \cite{Lunardi1997}, \cite{Pascucci2},
\cite{Francesco} and \cite{NyPasPol}): we note that 
some authors introduce only the definition of $C^{0,\a}_B$ and $C^{2,\a}_B$. Indeed, the
definition of $C^{1,\a}_B$ is technically more elaborate because it involves derivatives of
fractional (in the intrinsic sense) order and therefore is sometimes omitted.}

{In \cite{Manfredini} and \cite{Francesco}, $C^{0,\a}_B$ is defined as the space of functions that
are bounded and H\"older continuous with respect to the homogeneous group structure: precisely, a
function $u\in C^{0,\a}_B$ on a domain $\O$ of $\Rdd$ if
\begin{equation}\label{eq:def_ordinezeroaltri}
|u|_{\a,B,\Om}:= \sup_{z\in\Om}|u(z)| + \sup_{z,\z\in \Om \atop z\neq
\z}\frac{|u(z)-u(\z)|}{\|\z^{-1}\circ z\|^{\a}_B}<\infty.
\end{equation}
Note that by adopting this definition, the estimate for the remainder of the 0th order Taylor
polynomial trivially follows. Corollary \ref{ccaa} shows that definition
\eqref{eq:def_ordinezeroaltri} is basically equivalent to {Definition
\ref{def:C_alpha_spaces}-i)}.} Similarly, \cite{NyPasPol} define the following norm in the space
$C^{1,\a}_B$:
\begin{equation}\label{eq:def_norm_1}
 |u|_{1+\a,B,\Om}:=|u|_{\a,B,\Om} + \sum_{i=1}^{p_0} |\p_{x_i}u|_{\a,B,\Om} +
 \sup_{z,\z \in \Om \atop z\neq \z} \frac{|u(z)-T_1 u(\z,z)|}{\|\z^{-1}\circ z\|^{1+\a}_B}.
\end{equation}
Various definitions of the space $C^{2,\a}_B(\Om)$ are used in the literature. \cite{Manfredini}
requires bounded and H\"older continuous second order derivatives, while \cite{Francesco} and
\cite{NyPasPol} also require the function {$u$} and its first $p_0$ spatial derivatives to be
H\"older continuous. Precisely, \cite{Manfredini} introduces the norm
\begin{equation}\label{eq:def_norm_2_M}
|u|_{2+\a,B,\Om}^{(M)}:=\sup_{\Om}|u| + \sum_{i=1}^{p_0}\sup_{\Om}|\p_{x_i}u| +\sum_{i,j=1}^{p_0}
|\p_{x_i,x_j}u|_{\a,B,\Om} + |Yu|_{\a,B,\Om},
\end{equation}
while \cite{Francesco} and \cite{NyPasPol} define
\begin{equation}\label{eq:def_norm_2_A}
|u|_{2+\a,B,\Om}:=|u|_{\a,B,\Om} + \sum_{i=1}^{p_0} |\p_{x_i}u|_{\a,B,\Om} +\sum_{i,j=1}^{p_0}
|\p_{x_i,x_j}u|_{\a,B,\Om} + |Yu|_{\a,B,\Om}.
\end{equation}

\subsection{Examples of functions in %$C^{n,\alpha}_B$ and
$C^{n,\alpha}_{B,\text{loc}}$}\label{sec:examples} 
For comparison, we give some examples of functions with different intrinsic and Euclidean
regularity. We set $d=2$ and $B$ as in \eqref{eq:B_matrix_kolm_prot} corresponding to the
\emph{prototype} Kolmogorov operator in
\eqref{e5}. %Moreover we consider a domain $\Omega\subseteq\R\times \R^2$.
\begin{example}\label{ex:positive part}
Consider the function $u:\R\times \R^2\longrightarrow\R$ given by $u(t,x_1,x_2)=|x_2 -c|$, with
$c\in \R$. This function is particularly relevant for financial applications since it is {often}
related
to the payoff of %the
so-called Asian-style derivatives. Clearly $u$ is Lipschitz continuous in the
Euclidean sense, but intrinsically we have %more, that is
$u\in C^{1,1}_{B,\text{loc}}(\R\times
\R^2)$ because $\partial_{x_1} u \in C^{0,1}_{B,\text{loc}}(\R\times \R^2)$
and $u\in C^{2}_{Y,\text{loc}}(\R\times \R^2)$. Note that $u\notin C^{2,\alpha}_{B,\text{loc}}(\R\times \R^2)$ because $u$ is not
$Y$-differentiable in $x_2=c$: nevertheless a \eqref{eq:estim_tay_n}-like  estimate for $n=2$ and
$\alpha=1$ holds for two points $z,\z\in\R\times \R^2$ sharing the same time-component, i.e.
\begin{equation}
  \left|u(z)- u(\z)\right|\le |x_2-\xi_2| \leq \|\z^{-1}\circ z\|_{B}^{3} , \qquad z=(t,x),\ \z=(t,\xi)\in\R\times \R^2.
\end{equation}
\end{example}

\begin{example}\label{ex:positive part_2}
As a variant of the previous example let us consider the function $u:\R\times \R^2\longrightarrow\R$ given by $u(t,x_1,x_2)=|x_2 -c|^{\frac{3}{2}}$, with
$c\in \R$.
This time $u\in C^{1,1/2}$, that is differentiable with H\"older continuous derivatives in the
Euclidean sense, but intrinsically we have 
$u\in C^{2,1}_{B,\text{loc}}(\R\times \R^2)$ because $\partial_{x_1} u\equiv 0$ 
and
\begin{equation}
Yu(t,x_1,x_2) = \frac{3}{2}x_1 \, |x_2-c|^{\frac{1}{2}}\:\mathrm{sgn}(x_2-c)\in C^{0,1}_{B,\text{loc}}.
\end{equation}
Also in the present example the function shows higher intrinsic regularity than the Euclidean one.
\end{example}

\begin{example}
It is easy to check that any function of the form $u(t,x_1,x_2)=f(x_2-t x_1)$ is constant along the integral curves $e^{\delta Y}(z)=(t+\delta,x_1,x_2+\delta x_1)$ for any $z\in \Omega$.
Therefore, we have $Y^n u \equiv 0$ for any $n\in\N$. 
In this particular case, we have that $u\in C^{n,\alpha}_{B,\text{loc}} $ if and only if $u\in
C^{n,\alpha}_{\text{loc}} $ in the Euclidean sense.
\end{example}
\begin{example}
The following function belongs to $C^{2,\alpha}_{B,\text{loc}}$ but only to
$C^{0,\alpha}_{\text{loc}}$:
\begin{equation}
u(t,x_1,x_2)=
  \begin{cases}
    \frac{1}{\sqrt{2\pi x_1^4}}\int_{\R} \exp\Big(\hspace{-3pt}-\frac{(y-x_2)^2}{2 x_1^4}\Big) |y|\, \dd y & \text{if}\ x_1\neq 0 , \\
    |x_2|  & \text{if}\ x_1 = 0.
  \end{cases}
\end{equation}
Indeed $u$ is continuous and smooth on $\{x_1\neq 0 \}$; in particular, $u\in
C^{2,1}_{\text{loc}}(\{x_1\neq 0 \})$ and $u\in C^{2,1}_{B,\text{loc}}(\{x_1\neq 0 \})$. On the
plane $\{x_1=0\}$ the Euclidean derivative $\partial_{x_2} u$ does not exist in $x_2=0$ for any
$t$ and thus $u\notin C^{2,\alpha}_{\text{loc}}$ for any $\alpha\in (0,1]$. On the other hand,
$\partial_{x_1}u$, $\partial_{x_1 x_1}u$ and $Yu$ exist on $\{x_1=0\}$ and they are all equal to
$0$. In particular, we have $\partial_{x_1 x_1}u,Yu\in C^{1}_{Y,\text{loc}}$ and
$\partial_{x_1}u\in C^{2}_{Y,\text{loc}}$. Moreover, one can directly prove that $\partial_{x_1
x_1}u,Yu\in C^{1}_{\partial_{x_1},\text{loc}}$ and thus, $u\in C^{2,1}_{B,\text{loc}}$.
\end{example}

\subsection{Asymptotic expansions for ultra-parabolic operators and application to mathematical finance.
}\label{sec:asympt_exp} We briefly discuss the possibility of exploiting our result in order to
obtain asymptotic expansions for variable-coefficients ultra-parabolic operators of the type
\begin{equation}\label{e1_bis}
  \Lc=\frac{1}{2}\sum_{j=1}^{p_{0}} a_{ij}(t,x) \p_{x_{j}x_{j}}+\sum_{i=1}^{p_0} a_i(t,x) \p_{x_i}+\langle
  Bx,\nabla\rangle+\p_{t},\qquad
  t\in\R,\ {x\in\R^{d}},
\end{equation}
where $(a_{ij}(t,x))_{1\leq i,j\leq p_0}$ is a positive definite $p_0\times p_0$ matrix for any
$(t,x)\in \R\times \R^{d}$ and $B$ {is} as in \eqref{eq:B_blocks}.  In the elliptic case, i.e.
$p_0=d$, two of the authors had previously proposed a \emph{Gaussian perturbative method} to carry
out a closed-from approximation of the {solution} of the backward Cauchy problem
\begin{equation}\label{eq:cauchy_problem}
  \begin{cases}
    \Lc u(t,x)=0,\qquad &(t,x)\in [0,T[\,\times {\R^{d}}, \\
    u(t,x)=\varphi(x),\qquad &x\in \R^{d},
  \end{cases}
\end{equation}
for a given $T>0$ and a given terminal datum $\varphi$. For a recent and thorough description of
such approach the reader can refer to \cite{LPP4} for the uniformly parabolic case and to
\cite{PP_compte_rendu} for the locally-parabolic case. Roughly speaking, under mild assumptions on
the coefficients $a_{ij}$ and $a_i$, the authors proved a small-time asymptotic expansion of the
type
\begin{equation}\label{eq:asympt_exp}
 u(t,x)= u_0 (t,x) + \sum_{n=1}^{N} \Gc_n u_0(t,x)+R_N(t,x),\qquad N\in \N, 
\end{equation}
with
\begin{equation}\label{eq:asympt_exp_bis}
 R_N(t,x)= \text{O}\left((T-t)^{\frac{m_{a,N}+m_{\phi}}{2}}\right) \qquad \text{ as }T-t\to 0^{+}.
\end{equation}
%at any given order .
{Here $u_0$ is the solution of the \emph{heat-type} operator $\Lc_0$ obtained by freezing the
coefficients of $\Lc$ and $(\Gc_n)_{n\ge 1}$ is a family of differential operators acting on $x$,
{polynomial in $(T-t)$,} 
and dependent on the Taylor coefficients of the functions $a_{ij}$ and $a_i$. The positive exponents $m_{a,N}$ and $m_{\phi}$, determining the
asymptotic rate of convergence of the expansion, depend on the regularity of the coefficients
$a_{ij}$, $a_i$ and of the terminal datum $\varphi$ respectively. Typically, we have
$m_{a,N}={N+1}$ and $m_{\phi}=k+1$ if $a_{ij},a_{i}\in C^{N,1}$ and $\phi\in C^{k,1}$
respectively, in the classical Euclidean meaning. }

In light of the invariance properties
\eqref{eq:translation_invariance}-\eqref{eq:dilation_invariance} it seems reasonable that, when
trying to extend such results to the ultra-parabolic framework, intrinsic regularity should be
considered as opposed to classical one. Precisely, we {could} %can
perform our analysis by assuming the
coefficients $a_{ij},a_{i}\in C^{N,1}_B$, the terminal datum $\phi\in C^{k,1}_B$, and make use of
the intrinsic Taylor formula of Theorem \ref{th:main} to carry out an asymptotic expansion similar
to that in \eqref{eq:asympt_exp}-\eqref{eq:asympt_exp_bis}. Thus we {could} %can
obtain accurate
closed-form approximate solutions to the Cauchy problem \eqref{eq:cauchy_problem}. At the
best of our knowledge, such a general result for %\eqref{e1_bis}-like
ultra-parabolic operators is not available in the literature: clearly, such an extension could be
also performed by considering Euclidean regularity for the coefficients and the terminal datum.
However, the benefit in exploiting the intrinsic regularity is twofold:
\begin{itemize}
\item[i)] first, since the operators $\Gc_n$ in \eqref{eq:asympt_exp} depend on the Taylor coefficients of the functions $a_{ij}$ and $a_i$, it is convenient to use the
intrinsic Taylor polynomial, being the latter typically a projection of the Euclidean one. In this
way we avoid taking up terms in the expansion that do not improve the quality of the
approximation;
\item[ii)] {secondly, since the bound in \eqref{eq:asympt_exp_bis} also depends on the regularity of the datum $\phi$, we can prove a higher accuracy
for the approximation \eqref{eq:asympt_exp} since the intrinsic regularity of $\phi$ is typically
greater than the Euclidean regularity: this is the case, for instance, in financial applications
(see \eqref{payoff}).}
\end{itemize}

The interest for 
seeking approximate solutions for degenerate \eqref{e1_bis}-like operators is justified by their connection with \emph{stochastic differential equations} and by their vast impact in numerous applications. As a matter of example, we could consider the problem arising in mathematical finance of pricing path-dependent options of Asian style. To fix the ideas, let us denote by $X^1$ a risky asset following the stochastic differential equation
\begin{equation*}
  \dd X^1_{t}=%X^1_{t}\dd t+
  \s(X^1_{t})X^1_{t}\dd W_{t},
\end{equation*}
where $W$ is a one-dimensional standard real Brownian motion under the risk-neutral measure $\mathbb{Q}$. The averaging prices for an \emph{arithmetic Asian option} are described by the additional $\R^+$-valued state process $X^2$ satisfying
\begin{equation}
  \dd X^2_{t}=X^1_{t}\dd t.
\end{equation}
By usual arbitrage arguments and by standard Feynman-Kac representation formula, the price of a European Asian option with terminal payoff $\phi(X^1_T,X^2_T)$ is given by
\begin{equation}\label{eq:expect_asian_arithm}
A_t=\mathbb{E}[\phi(X^1_T,X^2_T)|\mathcal{F}_t]=u(t,X^1_t,X^2_t),
\end{equation}
where $u$ solves the Cauchy problem \eqref{eq:cauchy_problem} with
\begin{equation}\label{probCauchydegen}
\Lc=\frac{\sigma^{2}(x_1)x_1^{2}}{2}\,\partial_{x_1 x_1}+x_1\p_{x_2}+\p_{t}.
\end{equation}
Typical payoff functions are given by
\begin{equation}\label{payoff}
\begin{split}
 \phi_{\text{fix}}(x_1,x_2)=\left(\frac{x_2}{T}-K\right)^{+}\qquad&\text{(fixed strike arithmetic Call)},\\
 \phi_{\text{flo}}(x_1,x_2)=\left(x_1-\frac{x_2}{T}\right)^{+}\qquad&\text{(floating strike arithmetic Call)}.
\end{split}
\end{equation}
Under suitable regularity and growth conditions, existence and uniqueness of the solution to the
Cauchy problem \eqref{probCauchydegen} were proved in \cite{BarucciPolidoroVespri}.

Even in the standard Black\&Scholes model (i.e. $\sigma(\cdot)\equiv \sigma$), the arithmetic
average $X^2_t$ is not log-normally distributed and its distribution is not trivial to
analytically characterize. An integral representation was obtained by \cite{Yor1992a}. However,
the latter is not very relevant for the practical computation of the expectation in
\eqref{eq:expect_asian_arithm}. Therefore, several authors proposed diverse alternative approaches
to efficiently compute the prices of arithmetic Asian options. A short and incomplete list
includes \cite{Linetsky2004}, \cite{DewynneShaw2008}, \cite{Gobet2014475} and
\cite{FoschiPagliaraniPascucci2011}. In the latter, the authors obtained an explicit third order
approximation of the type \eqref{eq:asympt_exp} for the fundamental solution of $\Lc$ and thus for
the joint distribution
of the couple $(X^1_t,X^2_t)$, but did not prove any bound for the remainder $R_3$. {By means of the Taylor expansion of Theorem \ref{th:main}, it seems possible to derive a general
asymptotic expansion of type \eqref{eq:asympt_exp}-\eqref{eq:asympt_exp_bis} and prove rigorous
error bounds at any order $N$: we plan to pursue this research direction in a forthcoming paper.}

We would like to emphasize that, in particular, the accuracy would depend on the intrinsic
regularity of the payoff function. For instance, $\varphi_{\text{fix}}\in C^{1,1}_{B,\text{loc}}$
(see Example \ref{ex:positive part}) and the short-time asymptotic convergence would be of order
$(N+3)/2$. This is an interesting point because the same approximating technique would return a
slower asymptotic convergence, namely $(N+2)/2$, if we would only take into account the Euclidean
regularity of the payoff, i.e. $\varphi_{\text{fix}}\in C^{0,1}_{\text{loc}}$. Furthermore, it is
also interesting to observe that, if we assumed the coefficient $\sigma$ to be also dependent on
$x_2$, the partial derivatives w.r.t. $x_2$ would start appearing in the intrinsic Taylor
polynomial, and thus in the operators $\Gc_n$, only from $n=2$. Performing the expansion by means
of the classical Taylor polynomial would thus yield some additional terms that have no impact on
the asymptotic convergence.

\section{Preliminaries}\label{sec:prelim}
In this section we collect several results that are preliminary to the proof of Theorem
\ref{th:main}. We first recall that under the standing Assumption \ref{assume:horm}, the matrix
$B$ takes the form
\begin{equation}\label{eq:B_blockshom}
 B=\left(
 \begin{array}{ccccc}
 0_{p_{0}\times p_{0}}& 0_{p_{0}\times p_{1}}&\cdots& 0_{p_{0}\times p_{r-1}}& 0_{p_{0}\times p_{r}}\\
 B_1 & 0_{p_{1}\times p_{1}} &\cdots& 0_{p_{1}\times p_{r-1}} & 0_{p_{1}\times p_{r}} \\
 0_{p_{2}\times p_{0}} & B_2 &\cdots& 0_{p_{2}\times p_{r-1}}& 0_{p_{2}\times p_{r}} \\
 \vdots & \vdots &\ddots& \vdots&\vdots \\ 0_{p_{r}\times p_{0}} & 0_{p_{r}\times p_{1}} &\cdots& B_r& 0_{p_{r}\times p_{r}}
 \end{array}
 \right),
\end{equation}
where $0_{p_{i}\times p_{j}}$ is a $p_{i}\times p_{j}$ null block. We also recall notation
\eqref{e6} for $\bar{p}_{k}$ with $k=0,\dots,r$. As a direct consequence of
\eqref{eq:B_blockshom}, we have that for any $n\le r$
\begin{equation}\label{eq:B_blockshomn}
 B^{n}=\left(
 \begin{array}{ccccc}
 0_{\bar{p}_{n-1}\times p_{0}}& 0_{\bar{p}_{n-1}\times p_{1}}&\cdots& 0_{\bar{p}_{n-1}\times p_{r-n}}& 0_{\bar{p}_{n-1}\times \left(\bar{p}_{r}-\bar{p}_{r-n}\right)}\\
 \prod\limits_{j=1}^{n}B_{j} & 0_{p_{n}\times p_{1}} &\cdots& 0_{p_{n}\times p_{r-n}} & 0_{p_{n}\times \left(\bar{p}_{r}-\bar{p}_{r-n}\right)} \\
 0_{p_{n+1}\times p_{0}} & \prod\limits_{j=2}^{n+1}B_{j} &\cdots& 0_{p_{n+1}\times p_{r-n}}& 0_{p_{n+1}\times \left(\bar{p}_{r}-\bar{p}_{r-n}\right)} \\
 \vdots & \vdots &\ddots& \vdots&\vdots \\ 0_{p_{r}\times p_{0}} & 0_{p_{r}\times p_{1}} &\cdots& \prod\limits_{j=r-n+1}^{r}B_{j}& 0_{p_{r}\times \left(\bar{p}_{r}-\bar{p}_{r-n}\right)}
 \end{array}
 \right),
\end{equation}
where
  $$\prod_{j=1}^{n}B_{j}=B_{n}B_{n-1}\cdots B_{1}.$$
Moreover $B^{n}=0$ for $n>r$, so that
\begin{equation}\label{ande11}
  e^{\d B}=I_{d}+\sum_{h=1}^{r}\frac{B^{h}}{h!}\d^{h},
\end{equation}
where $I_{d}$ is the $d\times d$ identity matrix.

We also recall the notation \eqref{eq:sub_multi_index}: for any $x=(x_{1},\dots,x_{d})\in\R^{d}$
and $n=0,\dots,r$, we denote by $x^{{[n]}}\in\R^{d}$ {the projection of $x$ on $\{0
\}^{\bar{p}_{n-1}}\times \mathbb{R}^{p_n} \times \{0 \}^{d-\bar{p}_{n}}$}. Then $\R^{d}$ {can be represented as a} %is the
direct sum:
  $$\R^{d}=\bigoplus_{n=0}^{r} V_{n}, \qquad V_{n}:=\{x^{[n]}\mid x\in\R^{d}\},\quad n=0,\dots,r.$$
\begin{remark}\label{rem:linearmapping}
By \eqref{eq:B_blockshomn} it is clear that
\begin{equation}\label{ande1bis}
B^n v \in \bigoplus_{k=n}^{r} V_{k}, \qquad v\in \R^d,
\end{equation}
and if $v\in V_{0}$ then
\begin{equation}\label{ande1}
 B^{n}v\in V_{n},\qquad  n=0,\dots,r.
\end{equation}
More precisely, let us set
\begin{equation}
 \bar{B}_{n}=\left(
 \begin{array}{cc}
 0_{\bar{p}_{n-1}\times p_{0}}& 0_{\bar{p}_{n-1}\times (r-p_{0})}\\
 \prod\limits_{j=1}^{n}B_{j} & 0_{p_{n}\times (r-p_{0})}  \\
 0_{(\bar{p}_{r}-\bar{p}_{n})\times p_{0}} & 0_{(\bar{p}_{r}-\bar{p}_{n})\times (r-p_{0})}
 \end{array}
 \right),
 \end{equation}
where the $p_n\times p_0$ matrix $\prod\limits_{j=1}^{n}B_{j} $ has full rank. Then we have
 $$B^{n}v=\bar{B}_n v,{\qquad v\in V_0,}$$
and the linear application $\bar{B}_{n}:V_0\to V_n$ is surjective but, in general, not injective.
For this reason, for any $n=1,\cdots,r$, we define the subspaces $V_{0,n}\subseteq V_0$ as
 \begin{equation}
 V_{0,n}=\{ x\in V_0 | x_{j}=0 \ \forall  j\notin \Pi_{B,n}  \},
 \end{equation}
with $ \Pi_{B,n}$ being the set of the indexes corresponding to the first $p_n$ linear independent columns of $\prod\limits_{j=1}^{n}B_{j}$. 
It is now trivial that the linear map $$\bar{B}_{n}:V_{0,n}\to V_n$$ is also injective. 
Notice that
\begin{equation}\label{eq:inclusion_V_0_n}
V_{0,r}\subseteq V_{0,r-1} \subseteq \cdots \subseteq V_{0,1}\subseteq V_{0,0}:= V_0.
\end{equation}
\end{remark}

\subsection{Commutators and integral paths}
In this section we construct approximations of the integral paths of the commutators of the vector
fields $X_{1},\dots,X_{p_{0}}$ and $Y$ in \eqref{e3}. In the sequel we shall use the following
notations: for any $v\in \Rd$ we set
 $$Y_{v}^{(0)}=\sum_{i=1}^{d}v_{i}\p_{x_{i}}.$$
Hereafter we will always consider $v\in V_{0}$. {In such way %so that actually
$Y_{v}^{(0)}$ will be actually} %is
a linear
combination of $X_{1},\dots,X_{p_{0}}$. Moreover we define recursively
\begin{equation}\label{ande7b}
   Y_{v}^{(n)}=[Y_{v}^{(n-1)},Y]=Y_{v}^{(n-1)}Y-Y Y_{v}^{(n-1)},\qquad n\in\N.
\end{equation}

\begin{remark}%\label{}
By induction it is straightforward to show that for any $u\in C^{\infty}(\Rdd)$, we have
\begin{equation}\label{ande7}
   Y_{v}^{(n)}u = \langle B^{n}v,\nabla u\rangle, \qquad n\in \N,
\end{equation}
with $B^{n}v\in V_{n}$ by \eqref{ande1}.
\end{remark}

{When applied to functions in $C^{n,\a}_{B,\text{\rm loc}}$, operator $Y_{v}^{(n)}$ can be
interpreted as a composition of Lie derivatives. Indeed we have the following.
\begin{lemma}\label{lemm:commutators}
Let $n\in \N$ and $u\in C^{n,\alpha}_{B,\textrm{loc}}$. Then, for any $v\in V_0$ and $k\in\N\cup
\{0\}$ with $2k+1\leq n$, we have $Y^{(k)}_v u\in C^{n-2k -1,\alpha}_{B,\textrm{loc}}$.
\end{lemma}
\begin{proof}
If $k=0$, the thesis is obvious since, by assumption, $\partial_{x_i}u\in
C^{n-1,\alpha}_{B,\text{\rm loc}}$ for $i=1,\dots,p_{0}$. To prove the general case we proceed by
induction on $n$. If $n\le 2$ there is nothing to prove because we only have to consider the case
$k=0$. Fix now ${n}\geq 2$. We assume the thesis to hold for any $m\leq {n}$ and prove it true for
${n}+1$. We proceed by induction on $k$. We have already shown the case $k=0$.
Thus, we assume the statement to hold for $k\in\N\cup \{0\}$ with $2(k+1)+1\leq n+1 $ %: we assume the statement to hold for any $0\leq j\leq k$,
and we prove it true for $k+1$. Note that, by definition \eqref{ande7b} we clearly have
\begin{equation}
Y_{v}^{(k+1)}u=Y_{v}^{(k)}Yu-Y\,Y_{v}^{(k)}u,
\end{equation}
with $v\in V_0$. Then the thesis follows by inductive hypothesis and since, by definition of
intrinsic H\"older space, $Yu\in C^{{n}-1,\alpha}_{B,\textrm{loc}}$.
\end{proof}

Next we show how to approximate the integral curves of the commutators $Y_{v}^{(k)}$ by  using a
rather classical technique from control theory. For any $n\in\{0,\dots,r\}$, $z=(t,x)\in\Rdd$,
$\d\in\R$ and $v\in V_{0}$, we define iteratively the family of trajectories
$\left(\g_{v,\d}^{(n,k)}(z)\right)_{k=n,\dots,r}$ as
\begin{align}\label{ande2_bis}
  \g_{v,\d}^{(n,n)}(z) & = e^{\d^{2 n+1} Y^{(n)}_{v}}(z)=\big(t,x+\d^{2 n+1} B^n v\big),\\
  \g_{v,\d}^{(n,k+1)}(z) & =
  e^{-\d^{2}Y}\left(\g_{v,-\d}^{(n,k)}\left(e^{\d^{2}Y}\left(\g_{v,\d}^{(n,k)}(z)\right)\right)\right),\qquad n\leq k\leq r-1.\label{eq:gamma_k1_bis}
\end{align}
We also set
\begin{equation}\label{eq:def_gamma_minus1_k}
\g_{v,\d}^{(-1,k)}(z)= \g_{v,\d}^{(0,k)}(z),\qquad 
0\leq k\leq r.
\end{equation}

\begin{lemma}\label{andl1_bis}
For any $n\in\{0,\cdots,r\}$, $(t,x)\in\Rdd$, $\d\in\R$ and $v\in V_{0}$ we have
\begin{equation}\label{ande3}
  \g_{v,\d}^{(n,k)}(t,x)=\left(t, x+S_{n,k}(\d)v\right),\qquad k=n,\dots,r,
\end{equation}
where
\begin{equation}\label{ande3_bis}
 S_{n,n}(\d)=\d^{2n+1} B^n v\quad \text{ and }\quad
 S_{n,k}(\d)=(-1)^{k-n}\d^{2n+1} B^n \sum\limits_{h\in\N^{k-n}\atop |h|\le  r}\frac{(-B)^{|h|}}{h!}\d^{2|h|},\qquad
 k=n+1,\dots,r,
\end{equation}
with $|h|=h_{1}+\cdots+h_{k}$. 
\end{lemma}
\begin{proof}
Fix $n=0$ and proceed by induction on $k$. The case $k=n$ is trivial.
Now, assuming \eqref{ande3}-\eqref{ande3_bis} as inductive hypothesis and noting that $S_{k}(-\d)=- S_{k}(\d)$, we
have
\begin{align}\label{ande15}
   \g_{v,\d}^{(k+1)}(t,x) & =  e^{-\d^{2}Y}\Big(\g_{v,-\d}^{(n,k)}\Big(e^{\d^{2}Y}\Big(\g_{v,\d}^{(n,k)}(t,x)\Big)\Big)\Big)%\\
  % &
  =  e^{-\d^{2}Y}\Big(\g_{v,-\d}^{(n,k)}\Big(e^{\d^{2}Y}\Big(t,x+S_{n,k}(\d)v\Big)\Big)\Big)\\
   & =  e^{-\d^{2}Y}\Big(\g_{v,-\d}^{(n,k)}\Big(t+\d^{2},e^{\d^{2}B}\Big(x+S_{n,k}(\d)v\Big)\Big)\Big)%\\
 %  &
  =  e^{-\d^{2}Y}\Big(t+\d^{2},e^{\d^{2}B}\big(x+S_{n,k}(\d)v\big)-S_{n,k}(\d)v\Big)\hspace{-10pt}\\
   & =  \Big(t,e^{-\d^{2}B}\Big(e^{\d^{2}B}\big(x+S_{n,k}(\d)v\big)-S_{n,k}(\d)v\Big)\Big)%\\
  % &
   =  \Big(t,x+S_{n,k}(\d)v-e^{-\d^{2}B}S_{n,k}(\d)v\Big).
\end{align}
On the other hand, by \eqref{ande11} we have
\begin{align}
 x+S_{n,k}(\d)v-e^{-\d^{2}B}S_{n,k}(\d)v &=x+S_{n,k}(\d)v-\bigg(I_{d}+\sum_{j=1}^{r}\frac{(-B)^{j}}{j!}\d^{2j}\bigg)S_{n,k}(\d)v \\
& = x-\bigg(\sum_{j=1}^{r}\frac{(-B)^{j}}{j!}\d^{2j}\bigg)S_{n,k}(\d)v  = x+S_{n,k+1}(\d)v,
 \end{align}
and this concludes the proof.
\end{proof}
\begin{remark}\label{rem:rem_ste4}
Note that
\begin{equation}\label{ande4}
  S_{n,k}(\d)=\d^{2k+1}B^{k}+\tilde{S}_{n,k}(\d),\qquad n\leq k\leq r,
\end{equation}
with
  $$\tilde{S}_{n,n}(\d):=0 \quad \text{ and }\quad
  \tilde{S}_{n,k}(\d):=(-1)^{k-n}\d^{2n+1} B^n\sum\limits_{h\in\N^{k-n}\atop k-n<|h|\le
  r}\frac{(-B)^{|h|}}{h!}\d^{2|h|},\qquad k=n+1,\dots,r.$$
Then we deduce from \eqref{ande3} that
 \begin{equation}\label{eq:gamma_k}
 \g_{v,\d}^{(n,k)}(z)=\big(t,x+\d^{2 k+1} B^{k} v\big)+\big(0,\tilde{S}_{n,k}(\d)v\big),\qquad n\leq k.
 \end{equation}
It is important to remark that $ \tilde{S}_{n,r}(\d)=0$ and, by \eqref{ande1bis}, we have
\begin{equation}\label{ande6}
 \tilde{S}_{n,k}(\d)v\in \bigoplus_{j=k+1}^{r}V_{j},\qquad k=n,\dots,r;
\end{equation}
since $v\in V_{0}$, then by \eqref{ande1} we have 
\begin{equation}
\g_{v,\d}^{(n,n)}(z)=(t,x)+(0,\d^{2 n+1} B^{n} v), \quad\text{with } B^{n} v\in V_n.
\end{equation}
Thus, by using notation \eqref{eq:sub_multi_index}, for any $k=n,\dots,r$ we have
\begin{equation}\label{ande5}
  \left|\big(\tilde{S}_{n,k}(\d)v\big)^{[j]}\right|\le c_{B}|\d|^{2j+1}|v|,\qquad j=k+1,\dots,r,
  \quad \delta \in \R,
\end{equation}
where the constant {\it $c_{B}$ depends only on the matrix $B$}.
If $|v|=1$, \eqref{ande5} also implies
\begin{align}
\hspace{-30pt}\big\|\big(\g_{v,\d}^{(n,k)}(z)\big)^{-1}\circ z\big\|_{B}&=\big\| z^{-1}\circ \g_{v,\d}^{(n,k)}(z) \big\|_{B}=\norm{\big(  (t,x+\delta^{2k+1} B^k v)+(0,\tilde{S}_{n,k}(\d)v) \big)^{-1}\circ (t,x)}_{B}\\
\\ & =\norm{\big(  0,-\delta^{2k+1} B^k v - \tilde{S}_{n,k}(\d)v \big)}_{B}
= |-\delta^{2k+1} B^k v - \tilde{S}_{n,k}(\d)v|_B  \le  c_B |\delta|  .\label{normstep}
\end{align}
\end{remark}
Next we show how to connect two points in $\Rdd$ that only differ w.r.t. the spatial components {by
only moving along the the integral curves $\gamma^{(n,k)}$ previously defined.
\begin{lemma}\label{lem:connecting_curves}
Let $n\in \{0,\cdots,r \}$, $\z=(t,\xi)\in\R\times\Rd$, $y\in\bigoplus\limits_{k=n}^{r}V_{k}$ and the points $\z_{k}=(t,\xi_{k})$, for $k=n-1,\cdots,r$, defined as %follows. For $k=n$, set
\begin{equation}\label{eq:ste24}
\z_{n-1}:=\z,\qquad\qquad \z_{k}:=\g_{v_{k},\d_{k}}^{(n-1,k)}(\z_{k-1}),\quad   v_{k}=\frac{w_k}{\left| w_k\right|},\quad \ \d_{k}=\big| w_k \big|,\quad k\geq n,
\end{equation}
where $w_k$ is the only vector in $V_{0,k}\subseteq V_{0}$ such that
$B^{k}w_{k}=y^{[k]}+\xi^{[k]}-\xi_{k-1}^{[k]}$. Then:
\begin{enumerate}
\item[i)] for any $k\in\{n,\cdots, r \}$ we have:
\begin{equation}\label{eq:ste25}
\d_{k}\le c_{B}|y|_{B},\qquad \qquad \xi_{k}^{[j]} =\xi^{[j]}+ y^{[j]},\quad j=0,\dots,k. 
\end{equation}
Note that, in particular, $\z_{r}=\z+(0,y)$;
\item[ii)] there exists a positive constant \emph{$c_B$, only dependent on the matrix $B$}, such that
\begin{equation}\label{eq:ste34}
\norm{\z_{k}^{-1}\circ \z}_B \leq c_B |y|_B,
\end{equation}
for any $k=n,\cdots, r$ and $0\leq \delta \leq \delta_k$.
\end{enumerate}
\end{lemma}
\begin{proof}{We first prove i).}
The second identity in \eqref{eq:ste25} easily stems from \eqref{ande6} and by definition of $v_{k}$ and $\d_{k}$. We then focus on the first one.
By Remark \ref{rem:linearmapping}, it is easy to prove that
\begin{equation}\label{eq:ste26}
  \d_{k}\le
  c_{B}\big| \xi^{[k]}+y^{[k]}-\xi_{k-1}^{[k]}\big|^{\frac{1}{2k+1}}.
\end{equation}
Moreover, by \eqref{ande5} 
we get
\begin{align}\label{eq:ste27}
  \big|\xi_{k}^{[j]}-\xi_{k-1}^{[j]}| \le c_{B}\big|\d_k|^{2j+1}
  ,\qquad
  j=k+1,\dots,r.
\end{align}
We proceed by induction on $k$. For $k=n$ the thesis immediately follows by \eqref{eq:ste26}. We now fix $n \leq k\leq r-1$ and assume the estimate to hold for any $n\leq h \leq k$. By \eqref{eq:ste26} we have %\newpage
\begin{align}
\hspace{0pt} \d_{k+1}&\le c_{B}\big|\xi^{[k+1]}+y^{[k+1]}-\xi_{k}^{[k+1]}\big|^{\frac{1}{2(k+1)+1}}%\\
 %&
 \le c_{B}\big|y^{[k+1]}\big|^{\frac{1}{2(k+1)+1}}+c_{B}\sum_{h=n}^{k}\big| \xi^{[k+1]}_{h}-\xi_{h-1}^{[k+1]} \big|^{\frac{1}{2(k+1)+1}}\hspace{0pt}
\intertext{(by \eqref{eq:ste27})}
 &\le c_{B}\big|y^{[k+1]}\big|^{\frac{1}{2(k+1)+1}}+c_{B}\sum_{h=n}^{k}\d_{h},
\end{align}
and the thesis for $k+1$ follows by inductive hypothesis.

We now prove ii). {As first step we prove that
\begin{equation}
 \norm{(\g_{v_k,\d}^{(n-1,k)}(\z_{k-1})\big)^{-1}\circ \z_{k-1}}_B\leq c_B |y|_B.
\end{equation}
By equations \eqref{ande3}, \eqref{normstep} and \eqref{eq:ste25} we get
\begin{align}
 \norm{(\g_{v_k,\d}^{(n-1,k)}(\z_{k-1})\big)^{-1}\circ \z_{k-1}}_B &=
 \norm{(t,\x_{k-1}+S_{n-1,k}(\d)v_k)^{-1}\circ (t,\x_{k-1})}_B\\
 &=\norm{(0,\x_{k-1}-(\x_{k-1} +S_{n-1,k}(\d)v_k))}_B \\
 &= \norm{(0,-S_{n-1,k}(\d)v_k)}_B \leq %c_B|\d| \leq
 c_B \d_k \leq c_B |y|_B.
\end{align}
This estimate along with equations \eqref{normstep} and  \eqref{eq:ste25} allow us to conclude.
Precisely, applying the quasi-triangular inequality we get
\begin{align}
 \norm{\z_{k}^{-1}\circ \z}_B &\leq c_B\sum_{i=n}^{k} \norm{\z_{i}^{-1}\circ \z_{i-1}}_B %\\
%&
 \leq c_B |y|_B.
\end{align}
}
\end{proof}
}
{We conclude the section with the following remark that allows to control the homogeneous distance
between two points along the same integral curve of $Y$.
\begin{remark}
By \eqref{eq:def_curva_integrale_campo} and \eqref{eq:translation} we have
\begin{equation}\label{eq:ste12}
\norm{z^{-1}\circ e^{\delta Y} (z)}_{B}=\norm{(e^{\delta Y} (z))^{-1} \circ z }_{B} =
|\delta|^{\frac{1}{2}},\qquad z\in \R\times\R^d,\quad \delta\in\R,
\end{equation}
\end{remark}
}

\section{Proof of Theorem \ref{th:main}}\label{sec:proof}

Theorem \ref{th:main} will be proved by induction on $n$, through the following steps:
\begin{itemize}
\item {\bf Step 1}: Proof for $n=0$;
\item {\bf Step 2}: Induction from $2n$ to $2n+1$ for any $0\leq n\leq r$;
\item {\bf Step 3}: Induction from $2n+1$ to $2(n+1)$ for any $0\leq n\leq r-1$;
\item {\bf Step 4}: Induction from $n$ to $n+1$ for any $n\geq 2r+1$.
\end{itemize}
A brief explanation is needed: the proof of Theorem \ref{th:main} cannot be carried out by a
simple induction on $n$, due to the qualitative differences in the Taylor polynomials of different
orders. For instance, one could suppose the theorem to hold for $n=2$ and consider a function
$u\in C^{3,\a}_{B,\text{\rm loc}}$.  By the inclusion property
\begin{equation}
C^{3,\a}_{B,\text{\rm loc}}\subseteq C^{2,\a}_{B,\text{\rm loc}},
\end{equation}
all the derivatives of second $B$-order do exist, i.e.
\begin{equation}
Y^k\p_x^{\b}u \in C^{2-2k-|\b|_B,\a}_{B,\text{\rm loc}}, \qquad 2k+|\b|_B\leq 2.
\end{equation}
However, $T_{3}u$ also contains the derivatives of $B$-order equal to $3$. These are exactly
\begin{equation}
 \p_{x_i,x_j,x_k}u, \quad Y\p_{x_i}u \qquad 1\leq i,j,k\leq p_0,
 \end{equation}
whose existence is granted by definition of $C^{3,\a}_{B,\text{\rm loc}}$, and the Euclidean
derivatives $$ \p_{x_l}u, \qquad p_0 <l \leq \bar{p}_1, $$ whose existence must be proved, as it
is not trivially implied by definition of $C^{3,\a}_{B,\text{\rm loc}}$. In general, such problem
arises every time when defining the Taylor expansion of order $2n+1$, $n=1,\dots,r$, i.e. when the
Euclidean derivatives w.r.t. the variables of level $n$ appear for the first time in the Taylor
polynomial. This motivates the need to treat the inductive step from $2n$ to $2n +1$ in a separate
way and therefore the necessity for Step 2 and Step 3 in the proof. Eventually, Step 4 is
justified by the fact that, when $n\geq 2r+1$, the existence of the Euclidean partial derivatives
w.r.t. any variable has already been proved and thus the proof goes smoothly without any further
complication.

{We now try to summarize the main arguments on which the proof is based. Roughly speaking, in
order to prove the estimate \eqref{eq:estim_tay_n_loc} (and \eqref{eq:estim_tay_n}), we shall be
able to connect any pair of points $z,\zeta\in\R\times\Rd$ and to have a control of the increment
of $u$ along the connecting path. The definition of  $C^{n,\a}_{B,\text{\rm loc}}$ (and
$C^{n,\a}_{B}$) does only specify the regularity along the fields $Y$ and $(\partial_{x_i})_{1\leq
i\leq p_0}$, but does not give any a priori information about the regularity along all the other
Euclidean fields $(\partial_{x_i})_{p_0<i\leq d}$. It seems then clear that, when trying to
connect $z$ and $\zeta$, we cannot simply \emph{move} along the canonical directions $(e_i)_{1\leq
i\leq d}$. We shall indeed take advantage of Lemma \ref{lem:connecting_curves} in order to go from
$\z$ to $z$ by using the integral curves $\gamma^{(n,k)}$ and then control the increment of $u$
along the connecting paths by exploiting the estimates contained in Remark \ref{rem:rem_ste4}.

To simplify the exposition, for each point listed above, we will first prove the inductive step on the global version (Part 3) of Theorem \ref{th:main}, in order to keep the proof free from additional technicalities needed to prevent the possibility of the integral curves $\gamma^{(n,k)}$ to exit the domain $\Omega$.  At the end of the section we will sketch the main guidelines through which the proof of the local version (Part 1 and Part 2) will become a straightforward modification of the global one}. In order to prove the main theorem we will need to state three auxiliary results, which will be proved step by step along with Theorem \ref{th:main}.

\begin{proposition}\label{prop:complementary_new}
{Let $u\in C^{2n+1,\a}_{B}$ with $\alpha\in ]0,1]$ and $n\in\Nzero$ with $n\leq r$. Then, there
exist the Euclidean partial derivatives $\partial_{x_{i}}u\in C^{0,\alpha}_{B}$ for any
$\bar{p}_{n-1}< i \leq \bar{p}_n$ and
\begin{equation}\label{eq:drivative_commutators}
Y_{v^{(n)}_i}^{(n)} u(z) = \partial_{x_i}u(z),\qquad z\in \Rdd,
\end{equation}
with $\big(v^{(n)}_{i}\big)_{\bar{p}_{n-1}< i\leq \bar{p}_n}$ being the family of vectors such that $v^{(n)}_i\in V_{0,n}$ with $B^n v^{(n)}_i=e_{i}$. Note that such family of vectors is
univocally defined (see Remark \ref{rem:linearmapping}).}
\end{proposition}

{
\begin{proposition}\label{prop:complementary_alternative}
Let $\alpha\in ]0,1]$, $n\in\Nzero$ with $n\leq r$, $m\in\{0,1 \}$ and $u\in C^{2 n+m,\a}_{B%,\text{\rm loc}
}$. Then, for any  $\max \{n -1 ,0\}\leq  k\leq r$ %$n \leq  2 k+1\leq 2 r +1$
and $v\in V_{0,k}$ with $|v|=1$, we have:
\begin{equation}\label{eq:estim_tay_n_bis}
\Big|u\Big(\g_{v,\d}^{\left(n-1,k\right)}(z)\Big)-T_{2 n+m} u\Big(z,\g_{v,\d}^{\left(n-1,k\right)}(z)\Big) \Big|\le c_B \|u\|_{C^{2 n+m,\a}_{B}} |\d|^{2 n+m+\a}, \qquad z=(t,x)\in\Rdd,\quad \delta\in\R,%,\quad 0\leq k\leq r,v\in V_{0,k};
\end{equation}
where $c_B$ is a positive constant that only depends on $B$.
\end{proposition}
}
{
\begin{proposition}\label{prop:complementary_bis_alternative}
Let $\alpha\in ]0,1]$, $n\in\Nzero$ with $n\leq r$, $m\in\{0,1 \}$ and $u\in C^{2 n+m,\a}_{B%,\text{\rm loc}
}$. Then, %for any {$0\leq k\leq \left \lfloor{\frac{n}{2}}\right \rfloor$}
we have:
\begin{equation}\label{eq:estim_distanza_omog_n}
\left|u(t,x)-T_{2n+m} u\big( (t,x),(t,x+\xi) \big)\right| \le c_B \|u\|_{C^{2n+m,\a}_{B}} |\x|_{B}^{2n+m+\a}%,\qquad t\in\R,\quad x,\x\in\R^{d}
, \qquad (t,x)\in\R\times\Rd,\quad \xi\in \bigoplus_{j=0}^{n-1}V_{j},
\end{equation}
where $c_B$ is a positive constant that only depends on $B$.
\end{proposition}
}

Propositions \ref{prop:complementary_alternative} and \ref{prop:complementary_bis_alternative} are particular cases of the main theorem and are preparatory to its proof. We will also make repeated use of the following.
\begin{remark}\label{rem:eq:mean_value_Yn}
Let $n\in\N_0$, $m\in\{0,1 \}$ and $u\in C^{2n+m,\alpha}_{B%,\text{\rm loc}
}$. Then, by Definition
\ref{def:C_alpha_spaces}, we have $Y^{n} u\in C^{m+\alpha}_{Y%,\text{\rm loc}
}$. Therefore, by the
Euclidean mean-value theorem along the vector field $Y$, for any $z=(t,x)\in\Rdd$ and
$\delta\in\R$% small enough
, there exists $\bar{\delta}$  with $|\bar{\delta}|\leq |\delta|$ such that
\begin{equation}\label{eq:mean_value_Yn_loc}
 u\big(e^{\delta Y}(z)\big)-u(z)-\sum_{i=1}^{n} \frac{\d^i}{i!}Y^iu(z)= \d^{n}\left( Y^{n} u\big(e^{\bar{\delta}Y}(z)\big)- Y^{n}u(z)\right), \end{equation}
and thus, by Definition \ref{def:intrinsic_alpha_Holder} along with Assumption
\ref{assume:formal_degrees},
\begin{equation}\label{eq:mean_value_Yn_B}
 \Big| u\big(e^{\delta Y}(z)\big)-u(z)-\sum_{i=1}^{n} \frac{\d^i}{i!}Y^iu(z) \Big|\leq %c_B
 \|u\|_{C^{2n+m,\alpha}_{B}}|\d|^{n+\frac{m+\a}{2}},\qquad \delta\in\R,\quad z\in \R\times\R^d.
\end{equation}
\end{remark}

\subsection{Step 1}

Here we give the proofs for
\begin{enumerate}
\item[-] Proposition \ref{prop:complementary_alternative} for $n=0,$ $m=0$;
\item[-] 
Theorem \ref{th:main} {(Part 3)} for $n=0$.
\end{enumerate}
We start by recalling that:
\begin{equation}\label{eq:T0}
T_0 u\big(z,\z \big)=u(z),\qquad z,\z\in\Rdd.
\end{equation}
\begin{proof}[Proof of Propostion \ref{prop:complementary_alternative} for $n=0$, $m=0$.]${}$
\\
We prove the thesis by induction on $k$. For $k=0$ the estimate %\eqref{eq:estim_tay_n_loc_bis} (and
\eqref{eq:estim_tay_n_bis} %)
trivially follows by combining definitions
\eqref{eq:def_gamma_minus1_k} and \eqref{ande2_bis} with the assumptions $v\in V_{0}$, $|v|=1$ and %$u\in C^{\a}_{\p_{x_{i}},\text{\rm loc}}$ (and
$u\in
C^{\a}_{\p_{x_{i}}}$ %respectively)
for any $i=1,\dots,p_{0}$. 

We now assume the thesis to hold for $k\geq 0$ and we prove it true for $k+1$. We recall
\eqref{eq:gamma_k1_bis} and set
 $$z_{0}=z,\quad z_{1}=\g_{v,\d}^{(0,k)}(z_{0}),\quad z_{2}=e^{\d^{2}Y}\left(z_{1}\right),
 \quad z_{3}=\g_{v,-\d}^{(0,k)}(z_{2}),\quad z_{4}=e^{-\d^{2}Y}\left(z_{3}\right)=\g_{v,\d}^{(0,k+1)}(z)=\g_{v,\d}^{(-1,k+1)}(z).$$
Now, by triangular inequality we get
 $$\big|u\big(\g_{v,\d}^{(-1,k+1)}(z)\big)-u(z)\big|\le
 \sum_{i=1}^{4}\left|u(z_{i})-u(z_{i-1})\right|,$$
and thus, %\eqref{eq:estim_tay_n_loc_bis} (and
\eqref{eq:estim_tay_n_bis} %)
for $k+1$ follows from
the inductive hypothesis and from the assumption %$u\in C^{\a}_{Y,\text{\rm loc}}$ (and
$u\in
C^{\a}_{Y}$% respectively)
.
\end{proof}
We are now ready to prove Part 3 of Theorem \ref{th:main} for $n=0$.
\begin{proof}[Proof of Theorem \ref{th:main} {(Part 3)} for $n=0$.]  ${}$ \\
We first consider the particular case
$z=(t,x)$, $\zeta=(t,\xi)$, with $x,\x\in\R^{d}$. Precisely, we show that, if $u\in C^{0,\a}_{B}$ we have
\begin{equation}\label{eq:estim_distanza_omog_0}
\left|u(t,x)- u(t,\xi)\right| \le c_B \|u\|_{C^{0,\a}_{B}} |x-\x|_{B}^{\a},\qquad t\in\R,\quad x,\x\in\R^{d}.
\end{equation}
{By the triangular inequality, we obtain
 $$
  \left|u(t,x)- u(t,\xi)\right| \le \sum_{i=0}^{r} |u(\z_i) - u(\z_{i-1})|, %,\qquad t\in\R,\quad x,\x\in\R^{d}.
 $$
where the points $\zeta_{k}=(t,\xi_{k})$, for $k=-1,0,\cdots,r$, are  defined as in Lemma
\ref{lem:connecting_curves} by setting $n=0$ and $v=x-\xi$.
% where we have set $z_{-1}=z=(t,x)$.
The estimate %\eqref{eq:estim_distanza_omog_0_loc} (and
\eqref{eq:estim_distanza_omog_0} %)
 then stems from %\eqref{eq:estim_tay_n_loc_bis} (and
 \eqref{eq:estim_tay_n_bis} %respectively)
 with $n=0$, combined with \eqref{eq:ste25}.}

We now prove the general case. For any $z=(t,x),\z=(s,\xi)\in\Rdd$, by triangular inequality we get
\begin{align}
  \left|u(z)-u(\z)\right|&\le  |u(z)-u(e^{(t-s)Y}(\z))| +|u(e^{(t-s)Y}(\z))-u(\z)|  \\
                           &=    |u(t,x)-u(t,e^{(t-s)B}\x)|+|u(e^{(t-s)Y}(\z))-u(\z)|.\label{ande24}%\\
\end{align}
Now, 
to prove \eqref{eq:estim_tay_n}, we use \eqref{eq:estim_distanza_omog_0} to bound the first term in \eqref{ande24}, $u\in C_{Y}^{\alpha}$ to bound the second one, and we obtain
\begin{equation}
 \left|u(z)-u(\z)\right|\le c_B \|u\|_{C^{0,\a}_{B}}  \|\z^{-1}\circ z\|_{B}^{\a}, \qquad z=(t,x),\ \z=(s,\xi)\in\Rdd,
\end{equation}
which concludes the proof.
\end{proof}

\subsection{Step 2}
Throughout this section we fix $\bar{n}\in \{0,\cdots,r \}$ {and assume} to be holding true:
\begin{enumerate}
\item[-] Proposition \ref{prop:complementary_new} %to hold true
for any $0\leq n\leq \bar{n}-1$, if $\bar{n}\geq 1$;
\item[-] Theorem \ref{th:main} %and Propositions \ref{prop:complementary}, \ref{prop:complementary_bis}
%to hold true
for any $0\leq n\leq 2\bar{n}$.
\end{enumerate}
Then we prove:
\begin{enumerate}
\item[-] Propositions \ref{prop:complementary_alternative} and \ref{prop:complementary_bis_alternative} for $n=\bar{n},m=1$;
\item[-] Proposition \ref{prop:complementary_new} %to hold
for $n=\bar{n}$;
\item[-] Theorem \ref{th:main} {(Part 3)} %to hold
for $n=2\bar{n}+1$.
\end{enumerate}

This induction step has to be treated separately because we
cannot assume a priori the existence of the first order Euclidean partial derivatives w.r.t. the $\bar{n}$-th level variables%, i.e. we do not know a priori whether $\big(\partial_{\bar{p}_{n-1}+i}u\big)_{1\leq i\leq p_n}$ do exist.
. Therefore, we introduce the following alternative definition of $(2\bar{n}+1)$-th order
$B$-Taylor polynomial of $u$ that does not make explicit use of the derivatives
$\big(\partial_{\bar{p}_{\bar{n}-1}+i}u\big)_{1\leq i\leq p_{\bar{n}}}$: %{\blue and will be used
\begin{align}
\hspace{-30pt}
\bar{T}_{2{\bar{n}}+1} u(\z,z)&:= \sum_{{0\leq 2 k + |\beta|_B \leq 2{\bar{n}}+1 \atop \beta^{[{\bar{n}}]}=0}}\frac{1}{k!\,\beta!}%\bigg(  \prod_{i=0}^r \frac{1}{\big| \alpha^{[i]} \big|!}  \bigg)
\big(Y^k \partial_{\xi}^{\beta}u(\z)\big) (t-s)^k\big( x-e^{(t-s)B}\xi  \big)^{\beta}\\ &\quad +
{\sum_{i=\bar{p}_{{\bar{n}}-1}+1}^{\bar{p}_{\bar{n}}} } \big( Y_{v^{({\bar{n}})}_i}^{({\bar{n}})}
u(\z) \big)\, {\big(x-e^{B(t-s)}\xi\big)_{i} },\qquad\qquad z=(t,x),\ \zeta=(s,\xi)\in\Rdd,
\label{eq:def_Tayolor_n_bis}
\end{align}
with $\big(v^{({\bar{n}})}_{i}\big)_{\bar{p}_{{\bar{n}}-1}< i\leq \bar{p}_{\bar{n}}}$ being the family of vectors such that $v^{({\bar{n}})}_i\in V_{0,{\bar{n}}}$ with $B^{\bar{n}} v^{({\bar{n}})}_i=e_{i}$. %For any $n\notin\{1,3,\cdots,2r+1 \}$, we set $\bar{T}_{n} u(\z,z):=T_{n} u(\z,z)$.
%\end{definition}
\begin{remark}\label{rem:step2}
The {Taylor} polynomial $\bar{T}_{2{\bar{n}}+1}u$ is well-defined {for any} $u\in
C^{2{\bar{n}}+1,\alpha}_{B,\text{\rm loc}}$
. In fact, by Lemma \ref{lemm:commutators} we have
\begin{equation}\label{eq:ste13}
Y_{v^{({\bar{n}})}_i}^{({\bar{n}})}u\in C^{0,\alpha}_{B,\text{\rm loc}},\qquad
\bar{p}_{{\bar{n}}-1}< i\leq \bar{p}_{\bar{n}}.
\end{equation}
On the other hand, by using {the} inclusion of the spaces $C^{n,\alpha}_{B,\text{\rm loc}}$ and
the inductive hypothesis (Theorem \ref{th:main}, Part 1), the Euclidean derivatives $$
\partial_{\xi}^{\beta}u(\z) ,\qquad 0\leq  |\beta|_B \leq 2{\bar{n}}+1,\quad  \beta^{[{\bar{n}}]}=0,
$$
are well defined. Therefore, by combining the inductive hypothesis on Proposition \ref{prop:complementary_new} and  Lemma \ref{lemm:commutators}, we have
\begin{equation}\label{eq:ste1}
Y^k \partial_{\xi}^{\beta}u(\z) \in C^{2{\bar{n}}+1-2k -|\beta|_B,\alpha}_{B,\text{\rm loc}}
,\qquad 0\leq 2k+ |\beta|_B \leq 2{\bar{n}}+1,\quad  \beta^{[{\bar{n}}]}=0.
\end{equation}
In particular, by analogous arguments, if $u\in C^{2{\bar{n}}+1,\alpha}_{B}$ we have that
\begin{align}
&Y_{v^{({\bar{n}})}_i}^{({\bar{n}})}u\in C^{0,\alpha}_{B},\qquad    \bar{p}_{{\bar{n}}-1}< i\leq \bar{p}_{\bar{n}},\label{eq:ste32}\\
&Y^k \partial_{\xi}^{\beta}u(\z) \in C^{2{\bar{n}}+1-2k -|\beta|_B,\alpha}_{B}  ,\qquad 0\leq 2k+ |\beta|_B \leq 2{\bar{n}}+1,\quad  \beta^{[{\bar{n}}]}=0.\label{eq:ste4}
\end{align}
\end{remark}
\begin{remark}\label{rem:ste1}
By simple linear algebra arguments, it is also easy to show that for a given $\alpha\in [0,1]$, $n\in\{0,\cdots,r  \}$ and {$u\in C^{2n+1,\alpha}_{B}$}, %and $v\in V_{0,n}$,
we have
\begin{equation}\label{eq:ste10}
 {\sum_{i=\bar{p}_{{{n}}-1}+1}^{\bar{p}_{{n}}} } \left(Y_{v^{({n})}_i}^{({n})} u(\z)\right) {(B^{n}v)_{i} } = Y_{v}^{n} u(\z), \qquad \zeta\in\Rdd,\quad v\in V_{0,n}.
\end{equation}
\end{remark}

\subsubsection{Proof of Propositions \ref{prop:complementary_alternative} and \ref{prop:complementary_bis_alternative},  for $\bf{n={\bar{n}}}$ and $\bf{m=1}$}
We prove Propositions \ref{prop:complementary_alternative} and \ref{prop:complementary_bis_alternative} on $\bar{T}_{2{\bar{n}}+1} u$, for $n={\bar{n}}$ and $m=1$.  Note that, after proving Proposition \ref{prop:complementary_new} for $n=\bar{n}$, the two versions of the Taylor polynomials $\bar{T}_{2{\bar{n}}+1} u$ and $T_{2{\bar{n}}+1} u$ will turn out to be equivalent.
\begin{proof}[Proof of Proposition \ref{prop:complementary_alternative} for $n={\bar{n}},m=1$]${}$
\\
We assume {$u\in C^{2{\bar{n}}+1,\alpha}_{B%,\text{\rm loc}
}$} and we have to prove that for any  $\max \{\bar{n} -1 ,0\}\leq  k\leq r$, %$n \leq  2 k+1\leq 2 r +1$
 $v\in V_{0,k}$ with $|v|=1$, and $z=(t,x)\in \R\times \R^d$, we have
\begin{equation}\label{eq:estim_tay_n_loc_ter}
u\Big(\g_{v,\d}^{\left(\bar{n}-1,k\right)}(z)\Big)=\bar{T}_{2 \bar{n}+1} u\Big(z,\g_{v,\d}^{\left(\bar{n}-1,k\right)}(z)\Big) + R^{(\bar{n}-1,k)}_{v,\delta}(z),%\quad \text{as } \delta \to 0.%\quad k\geq 0.
\end{equation}
with
\begin{equation}\label{eq:estim_tay_n_ter}
|R^{(\bar{n}-1,k)}_{v,\delta}(z)|\le c_B \|u\|_{C^{2 \bar{n}+1,\a}_{B}} |\d|^{2 \bar{n}+1+\a}, \qquad z=(t,x)\in\Rdd,\quad \delta\in\R.%,\quad 0\leq k\leq r,v\in V_{0,k};
\end{equation}
We prove %\eqref{eq:estim_tay_n_loc_ter} and
\eqref{eq:estim_tay_n_ter} by induction on $k$. %We prove the thesis for $k=\max \{\bar{n} -1 ,0\}$.

\noindent {\bf Proof for ${k=\max \{\bar{n} -1 ,0\}}$:} because of the particular definition of
$\g_{v,\d}^{({n},k)}$ we have to treat separately the cases $\bar{n}=0$, $\bar{n}=1$ and
${\bar{n}}>1$.

\noindent \underline{Case ${\bar{n}}=0$}: by \eqref{eq:def_gamma_minus1_k} and \eqref{ande2_bis}
we have $$ \g_{v,\d}^{\left(-1,0\right)}(z)=u(t,x+\d v), $$ and thus, by
\eqref{eq:def_Tayolor_n_bis}, \eqref{eq:estim_tay_n_loc_ter} for $k=0$ reads as
\begin{equation}\label{eq:ste7}
u(t,x+\d v)=u(t,x) + \d \sum_{i=1}^{p_0}\p_{x_{i}}u(t,x)v_{i} + R^{(-1,0)}_{v,\delta}(z). %O\big(|\d|^{2{\bar{n}}+1+\a}\big),\quad \text{as } \delta \to 0.%\quad k\geq 0.
\end{equation}
Now, by the %multi-dimensional Euclidean
{standard mean-value theorem}, there exist $(\bar{v}_i)_{i=1,\dots,p_0}$ with  $\bar{v}_i\in V_0$
and $|\bar{v}_i|\leq |v| \le 1$, such that
\begin{equation}\label{eq:ste8}
u(t,x+\d v)-u(t,x)= \d \sum_{i=1}^{p_0}\p_{x_{i}}u(t,x+\d \bar{v}_i)v_{i},
\end{equation}
and thus {\begin{align}
  R^{(-1,0)}_{v,\delta}(z) &=\d \sum_{i=0}^{p_0}(\p_{x_i}u(t, x+\d \bar{v}_i)-\p_{x_i}u(t,x)) v_i .
\end{align}
Note that $\p_{x_i}u \in C^{0,\alpha}_{B}$ for any $1\leq i \leq p_0$ {because} $u\in
C^{1,\alpha}_{B}$ by assumption. Therefore
estimate \eqref{eq:estim_tay_n_ter} %for $\bar{n}=0$ and $k=0$
stems from Part 3 of Theorem \ref{th:main} for $n=0$.
}

\smallskip\noindent\underline{Case ${\bar{n}}=1$}: by \eqref{ande2_bis} we have $$
\g_{v,\d}^{\left(0,0\right)}(z)=u(t,x+\d v), $$ and thus, by \eqref{eq:def_Tayolor_n_bis},
\eqref{eq:estim_tay_n_loc_ter} for $k=0$ reads as
\begin{align}\label{eq:ste5}
u(t,x+\d v)&=u(t,x) + \d \sum_{i=1}^{p_0}\p_{x_{i}}u(t,x)v_{i}+
\frac{\d^2}{2!}\sum_{i,j=1}^{p_0}\p_{x_i x_j}u(t,x) v_i v_j
+\frac{\d^3}{3!}\sum_{i,j,l=1}^{p_0}\p_{x_i x_j x_l}u(t,x) v_i v_j v_l + R^{(0,0)}_{v,\delta}(z). %O\big(|\d|^{2{\bar{n}}+1+\a}\big),\quad \text{as } \delta \to 0.%\quad k\geq 0.
\end{align}
Now, by the mean-value theorem, there exist $(\bar{v}_{i,j,k})_{1\leq i,j,k \leq p_0}$, with
$\bar{v}_{i,j,k}\in V_0$ and $|\bar{v}_{i,j,k}|\leq |v|\leq 1$, such that
\begin{align}\label{eq:ste6}
u(t,x+\d v)-u(t,x) - \d \sum_{i=1}^{p_0}\p_{x_{i}}u(t,x)v_{i}-
\frac{\d^2}{2!}\sum_{i,j=1}^{p_0}\p_{x_i x_j}u(t,x) v_i v_j&=
%\\
%&\quad +
\frac{\d^3}{3!}\sum_{i,j,l=1}^{p_0}\p_{x_i x_j x_l}u(t,x+\d \bar{v}_{i,j,k}) v_i v_j v_l , %O\big(|\d|^{2{\bar{n}}+1+\a}\big),\quad \text{as } \delta \to 0.%\quad k\geq 0.
\end{align}
and thus {\begin{align}
  R^{(0,0)}_{v,\delta}(z) &= \frac{\delta^3}{3!} \sum_{i,j,l=1}^{p_0} \big(\partial_{x_i,x_j,x_l}u(t,x+\d \bar{v}_{i,j,l})-\partial_{x_i,x_j,x_l}u(t,x)\big) v_i v_j v_l.
\end{align}
Note that $\p_{x_i,x_j,x_l}u \in C^{0,\alpha}_{B}$ for any $1\leq i,j,l \leq p_0$ since, by
assumption, $u\in C^{3,\alpha}_{B}$. Estimate \eqref{eq:estim_tay_n_ter} then stems from Part 3 of
Theorem \ref{th:main} for $n=0$.
}

\smallskip\noindent\underline{Case ${\bar{n}}>1$}:
by \eqref{ande2_bis} we have $$
\g_{v,\d}^{\left(\bar{n}-1,\bar{n}-1\right)}(z)=u(t,x+\d^{2{\bar{n}}-1} B^{{\bar{n}}-1}v), $$ and
thus, by \eqref{eq:def_Tayolor_n_bis}, \eqref{eq:estim_tay_n_loc_ter} for $k=\bar{n}-1$ reads as
\begin{equation}\label{eq:ausiliar_general_loc_2}
u\big(t,x+\d^{2{\bar{n}}-1} B^{{\bar{n}}-1}v\big)=u(t,x) + \d^{2{\bar{n}}-1}\sum_{i=\bar{p}_{{\bar{n}}-2}+1}^{\bar{p}_{{\bar{n}}-1}}\p_{x_{i}}u(t,x)(B^{{\bar{n}}-1}v)_{i} + R^{(\bar{n}-1,\bar{n}-1)}_{v,\delta}(z). %O\big(|\d|^{2{\bar{n}}+1+\a}\big),\quad \text{as } \delta \to 0.%\quad k\geq 0.
\end{equation}
Now, by the mean-value theorem, there {exists} a family of vectors
$(\bar{v}_i)_{\bar{p}_{{\bar{n}}-2}< i\leq \bar{p}_{{\bar{n}}-1}}$, with  $\bar{v}_i \in V_{\bar{n}-1}$ and %$|B^{{\bar{n}}-1}\bar{v}_i|\leq |B^{{\bar{n}}-1}v|$
$|\bar{v}_i|\leq |B^{{\bar{n}}-1}v|\leq c_B$, such that
\begin{equation}
u(t,x+\d^{2{\bar{n}}-1} B^{{\bar{n}}-1}v)-u(t,x)=\d^{2{\bar{n}}-1} \sum_{i=\bar{p}_{{\bar{n}}-2}+1}^{\bar{p}_{{\bar{n}}-1}}\p_{x_{i}}u(t,x+\d^{2{\bar{n}}-1} %B^{{\bar{n}}-1}\bar{v}_i
\bar{v}_i)(B^{{\bar{n}}-1}v)_i,
\end{equation}
and thus, {\begin{align}
  R^{(\bar{n}-1,\bar{n}-1)}_{v,\delta}(z)  &=\d^{2{\bar{n}}-1} \sum_{i=\bar{p}_{{\bar{n}}-2}+1}^{\bar{p}_{{\bar{n}}-1}}\left(\p_{x_{i}}u(t,x+\d^{2{\bar{n}}-1}
  \bar{v}_i)-\p_{x_{i}}u(t,x)\right)(B^{{\bar{n}}-1}v)_i\\
  & =\d^{2{\bar{n}}-1} \sum_{i=\bar{p}_{{\bar{n}}-2}+1}^{\bar{p}_{{\bar{n}}-1}}\Big(\p_{x_{i}}u(t,x+\d^{2{\bar{n}}-1} %B^{{\bar{n}}-1}\bar{v}_i
  \bar{v}_i)-T_2 (\p_{x_{i}}u)\big( (t,x),  (t,x+\d^{2{\bar{n}}-1} %B^{{\bar{n}}-1}\bar{v}_i
  \bar{v}_i) \big)\Big)(B^{{\bar{n}}-1}v)_i.\label{eq:ste3}
\end{align}
Now, by \eqref{eq:ste4} in Remark \ref{rem:step2}, we have $\p_{x_{i}}u \in C^{2,\alpha}_{B} $ for
any $\bar{p}_{{\bar{n}}-2}<i\leq \bar{p}_{{\bar{n}}-1}$. Therefore estimate
\eqref{eq:estim_tay_n_ter} stems from Part 3 of Theorem \ref{th:main} for $n=2$.
}

\medskip\noindent{\bf Inductive step on $k$}:
we assume the thesis to hold true for a fixed $ \max \{\bar{n} -1 ,0\}\leq k< r$ and prove it true
for $k+1$. Consider thus $v\in V_{0,k+1}$ with $|v|=1$. Set $$
\tilde{T}_{2{\bar{n}}+1}u(\z,z)=\bar{T}_{2{\bar{n}}+1} u(\z,z)-u(\z),\qquad
z,\zeta\in\R\times\R^d, $$ and
\begin{equation}\label{eq:ste11}
z_{0}=z,\quad z_{1}=\g_{v,\d}^{\left(\bar{n}-1,k\right)}(z_{0}),\quad z_{2}=e^{\d^{2}Y}\left(z_{1}\right),
 \quad z_{3}=\g_{v,-\d}^{\left(\bar{n}-1,k\right)}(z_{2}),\quad z_{4}=e^{-\d^{2}Y}\left(z_{3}\right)=\g_{v,\d}^{\left(\bar{n}-1,k+1\right)}(z).
\end{equation}
According to this notation we have
\begin{equation}\label{eq:mic1_bis}
 R^{(\bar{n}-1,k+1)}_{v,\delta}(z)= u\Big(\g_{v,\d}^{\left(\bar{n}-1,k+1\right)}(z)\Big)-\bar{T}_{2 \bar{n}+1} u\Big(z,\g_{v,\d}^{\left(\bar{n}-1,k+1\right)}(z)\Big) = u(z_4)-\bar{T}_{2{\bar{n}}+1}u(z_0,z_4)= \sum_{i=1}^6 G_i,
\end{equation}
with
\begin{align}%\label{eq:mic1}
&  G_1=  u(z_4)-u(z_3) -\sum_{i=1}^{\bar{n}} \frac{(-\d^2)^i}{i!}Y^iu(z_3),\qquad\qquad\qquad\qquad G_2= u(z_3)-u(z_2)-\tilde{T}_{2{\bar{n}}+1}u(z_2,z_3),
  \\
& G_3 = \sum_{i=1}^{\bar{n}} \frac{(-\d^2)^i}{i!}Y^iu(z_2) +u(z_2)-u(z_1),\qquad\qquad\qquad\qquad G_4 = \tilde{T}_{2{\bar{n}}+1}u(z_1,z_0) +u(z_1)-u(z_0),\\
& G_5 = \sum_{i=1}^{\bar{n}} \frac{(-\d^2)^i}{i!}\left(Y^i u(z_3)-Y^iu(z_2)-\tilde{T}_{2(\bar{n}-i)+1}Y^i u(z_2,z_3)\right), \\
& G_6 = \tilde{T}_{2{\bar{n}}+1}u(z_2,z_3)- \tilde{T}_{2{\bar{n}}+1}u(z_1,z_0)-\tilde{T}_{2{\bar{n}}+1}u(z_0,z_4)+ \sum_{i=1}^{\bar{n}} \frac{(-\d^2)^i}{i!}\tilde{T}_{2(\bar{n}-i)+1}Y^i u(z_2,z_3).
\end{align}
Now, by applying Remark \ref{rem:eq:mean_value_Yn} with $n=\bar{n}$, $m=1$, on $G_1$ and $G_3$,
and by using the inductive hypothesis on $G_2$ and $G_4$ (note that by \eqref{eq:inclusion_V_0_n}
$V_{0,k+1}\subseteq V_{0,k}$), we have
 $$
  |G_1+G_2+G_3+G_4| \le c_B \|u\|_{C^{2 \bar{n}+1,\a}_{B}} |\d|^{2 \bar{n}+1+\a}, \qquad
  z=(t,x)\in\Rdd,\quad \delta\in\R.
 $$
To bound $G_5$, it is enough to observe that, by Definition {\ref{def:C_alpha_spaces}, {$u\in
C^{2\bar{n}+1,\alpha}_{B}$} implies $Y^i u \in C^{2(\bar{n}-i)+1,\alpha}_{B}$}, for any
$i=1,\cdots, \bar{n}$. Therefore, the bound follows by applying Part 3 of Theorem \ref{th:main}
for $n=2(\bar{n}-i)+1$, combined with \eqref{normstep}.

In order to estimate $G_6$ and conclude the proof, we need to distinguish on whether
$k=\max\{\bar{n}-1,0\}$, $k=\bar{n}$ or $k>\bar{n}$.

\smallskip\noindent\underline{Case $k>{\bar{n}}$}: there is nothing to prove because, by definitions \eqref{eq:def_Tayolor_n_bis} and
\eqref{eq:ste11}, we have  $G_6\equiv 0$.

\smallskip\noindent\underline{Case $k={\bar{n}}$}: first note that, in this case, the term $G_6$ reduces to
\begin{equation}
G_6 = \tilde{T}_{2{\bar{n}}+1}u(z_2,z_3)- \tilde{T}_{2{\bar{n}}+1}u(z_1,z_0) = \tilde{T}_{2{\bar{n}}+1}u\big(z_2,\g_{v,-\d}^{\left(\bar{n}-1,\bar{n}\right)}(z_{2})\big)- \tilde{T}_{2{\bar{n}}+1}u\big(z_1,\g_{v,-\d}^{\left(\bar{n}-1,\bar{n}\right)}(z_{1})\big),
\end{equation}
and by definition \eqref{eq:def_Tayolor_n_bis}, along with \eqref{eq:gamma_k}-\eqref{ande6}, we
get {\begin{align} |G_6| & = \Big| \delta^{2\bar{n}+1}
{\sum_{i=\bar{p}_{{\bar{n}}-1}+1}^{\bar{p}_{\bar{n}}} } \Big( Y_{v^{({\bar{n}})}_i}^{({\bar{n}})}
u(z_1) - Y_{v^{({\bar{n}})}_i}^{({\bar{n}})} u(z_2) \Big)\, \big(B^{\bar{n}}v\big)_{i} \Big|=
 \intertext{(by Remark \ref{rem:ste1} with $n=\bar{n}$ and since $v\in
V_{0,{\bar{n}}+1}\subseteq V_{0,{\bar{n}}}$)} & =\Big| \delta^{2\bar{n}+1} \big(
Y_{v}^{({\bar{n}})} u(z_1) - Y_{v}^{({\bar{n}})} u(z_2) \big) \Big|\le  \intertext{(by hypothesis
$u\in C^{2\bar{n}+1,\alpha}_{B}$ and thus, by Lemma \ref{lemm:commutators},
$Y_{v}^{({\bar{n}})}u\in C^{0,\alpha}_{B}\subseteq C^{\alpha}_{Y}$)} &\leq c_B \|u\|_{C^{2
\bar{n}+1,\a}_{B}} |\d|^{2 \bar{n}+1+\a}.
\end{align}}

\noindent\underline{Case $k=\max\{\bar{n}-1,0\}$}: we only need to prove the case $\bar{n}>0$. We first consider $\bar{n}\geq 2$. We have
\begin{align}
G_6 %&= \tilde{T}_{3}u(z_2,z_3)- \tilde{T}_{3}u(z_1,z_0)-\tilde{T}_{3}u(z_0,z_4)-\d^2 \tilde{T}_{1}Y u(z_2,z_3)\\
&= \tilde{T}_{2\bar{n}+1}u\big(z_2,\g_{v,-\d}^{\left(\bar{n}-1,\bar{n}-1\right)}(z_{2})\big)- \tilde{T}_{2\bar{n}+1}u\big(z_1,\g_{v,-\d}^{\left(\bar{n}-1,\bar{n}-1\right)}(z_{1})\big)-\tilde{T}_{2\bar{n}+1}u\big(z_0,\g_{v,\d}^{\left(\bar{n}-1,\bar{n}\right)}(z_{0})\big)\\
&\quad + \sum_{i=1}^{\bar{n}} \frac{(-\d^2)^i}{i!}\tilde{T}_{2(\bar{n}-i)+1}Y^i u\big(z_2,\g_{v,-\d}^{\left(\bar{n}-1,\bar{n}-1\right)}(z_{2})\big).
\end{align}
Now recall that, by \eqref{eq:gamma_k}-\eqref{ande6},
\begin{align}
\g_{v,-\d}^{\left(\bar{n}-1,\bar{n}-1\right)}(z)&=\big(  t, x - \delta^{2(\bar{n}-1)+1} B^{\bar{n}-1} v  \big),\\
 \g_{v,\d}^{\left(\bar{n}-1,\bar{n}\right)}(z)&=\big(  t, x + \delta^{2\bar{n}+1} B^{\bar{n}} v  +\tilde{S}_{\bar{n}-1,\bar{n}}(\d)v  \big),\qquad \tilde{S}_{\bar{n}-1,\bar{n}}(\d)v\in \bigoplus_{j=\bar{n}+1}^{r}V_{j},
\end{align}
and thus, by definition \eqref{eq:def_Tayolor_n_bis}, %along with \eqref{eq:gamma_k}-\eqref{ande6} and Remark \ref{rem:ste1} for $n=\bar{n}$,
we obtain
\begin{align}
 G_6 &=  \delta^{2(\bar{n}-1)+1} {\sum_{i=\bar{p}_{\bar{n}-2}+1}^{\bar{p}_{\bar{n}-1}} \big( \partial_{x_{i}} u(z_1) - \partial_{x_{i}} u(z_2) +\delta^2 \partial_{x_{i}} Y u(z_2) \big)\, (B^{\bar{n}-1} v)_i %\\
%& \quad
-  \delta^{2\bar{n}+1} \sum_{i=\bar{p}_{{\bar{n}}-1}+1}^{\bar{p}_{\bar{n}}}}
Y^{(\bar{n})}_{v^{(\bar{n})}_i} u(z_0) \,(B^{\bar{n}} v)_i=
 \intertext{(by Proposition \ref{prop:complementary_new} for $n=\bar{n}-1$)}
 &=  \delta^{2(\bar{n}-1)+1} {\sum_{i=\bar{p}_{\bar{n}-2}+1}^{\bar{p}_{\bar{n}-1}}} \big( Y^{(\bar{n}-1)}_{v^{(\bar{n}-1)}_i} u(z_1) - Y^{(\bar{n}-1)}_{v^{(\bar{n}-1)}_i} u(z_2) +\delta^2 Y^{(\bar{n}-1)}_{v^{(\bar{n}-1)}_i} Y u(z_2) \big) \, (B^{\bar{n}-1} v)_i \\
 &\quad  -  \delta^{2\bar{n}+1} {\sum_{i=\bar{p}_{{\bar{n}}-1}+1}^{\bar{p}_{\bar{n}}}} Y^{(\bar{n})}_{v^{(\bar{n})}_i} u(z_0)(B^{\bar{n}} v)_i=%   \hspace{-40pt}
\intertext{(by applying Remark \ref{rem:ste1} with $n=\bar{n}-1$ and $n=\bar{n}$, and since $v\in
V_{0,\bar{n}}\subseteq V_{0,\bar{n}-1}$)} &=\delta^{2(\bar{n}-1)+1}  \big( Y^{(\bar{n}-1)}_{v}
u(z_1) - Y^{(\bar{n}-1)}_{v} u(z_2) +\delta^2 Y^{(\bar{n}-1)}_{v} Y u(z_2)\big) -  \delta^{2
\bar{n}+1}   Y^{\bar{n}}_{v} u(z_0)= \intertext{(since, by definition \eqref{ande7b},
$Y^{\bar{n}-1}_{v}Y   =  Y^{\bar{n}}_{v} + YY^{\bar{n}-1}_{v} $)} &= \delta^{2(\bar{n}-1)+1} \big(
Y^{(\bar{n}-1)}_{v} u(z_1) - Y^{(\bar{n}-1)}_{v} u(z_2) +\delta^2  Y
Y^{(\bar{n}-1)}_{v}u(z_2)\big) +  \delta^{2 \bar{n}+1} \big(Y^{\bar{n}}_{v} u(z_2) -
Y^{\bar{n}}_{v} u(z_0) \big)=\sum_{i=1}^3 F_i.
\end{align}
with
\begin{align}
 F_1 &= \delta^{2(\bar{n}-1)+1}  \left(Y^{(\bar{n}-1)}_{v} u(z_1) - Y^{(\bar{n}-1)}_{v} u(z_2)+\delta^2  Y Y^{(\bar{n}-1)}_{v}u(z_2)\right),\\
 F_2 &= \delta^{2 \bar{n}+1}\left(Y^{\bar{n}}_{v} u(z_2) - Y^{\bar{n}}_{v} u(z_1)\right),\qquad\qquad
 F_3 = \delta^{2 \bar{n}+1}\left(Y^{\bar{n}}_{v} u(z_1) - Y^{\bar{n}}_{v} u(z_0)\right).
\end{align}
{Now, to bound $F_1$ it is sufficient to note that, by Lemma \ref{lemm:commutators}, %$Y_{v}^{({\bar{n}}-1)}u\in C^{2,\a}_{B,\text{\rm loc}}$ (
$Y_{v}^{({\bar{n}}-1)}u\in C^{2,\a}_{B}$ and thus the bounds directly follow by applying Remark
\ref{rem:eq:mean_value_Yn} with $n=1$ and $m=0$. To bound the terms $F_2$ and $F_3$ we
use that, by Lemma \ref{lemm:commutators}, %$Y^{\bar{n}}_{v} u\in C^{0,\a}_{B,\text{\rm loc}}$ (
$Y^{\bar{n}}_{v}u\in C^{0,\a}_{B}$%)
. The estimate for $F_3$ then follows by %Part 2 (
Part 3 %)
of Theorem \ref{th:main} for $n=0$ along with equation \eqref{normstep}, whereas the one for $F_2$ is
 a consequence of the inclusion %$ C^{0,\a}_{B,\text{\rm loc}}\subseteq C^{\a}_{Y,\text{\rm loc}}$ (
$C^{0,\a}_{B}\subseteq C^{\a}_{Y}$ %)
and of Remark \ref{rem:eq:mean_value_Yn}.
}

Finally, the case $\bar{n}=1$ is analogous, but $G_6$ contains two more terms:
\begin{align}
F_4 &=  \frac{\d^2}{2!} \sum_{i,j=1}^{p_0} \big( \partial_{x_i,x_j}u(z_2)-\partial_{x_i,x_j}u(z_1)\big)v_i v_j ,
\\
F_5 &=- \frac{\d^3}{3!} \sum_{i,j,l=1}^{p_0} \big( \partial_{x_i,x_j,x_l}u(z_2)-\partial_{x_i,x_j,x_l}u(z_1)\big)v_i v_j v_l ,
\end{align}
which can be estimated by {using that %$\partial_{x_i,x_j,x_l}u \in C^{0,\a}_{B,\text{\rm loc}}\subseteq C^{\a}_{Y,\text{\rm loc}}$ (
$\partial_{x_i,x_j,x_l}u\in C^{0,\a}_{B}\subseteq C^{\a}_{Y}$ and %$\partial_{x_i,x_j}u \in C^{1,\a}_{B,\text{\rm loc}}\subseteq C^{\a+1}_{Y,\text{\rm loc}}$ (
$ \partial_{x_i,x_j}u\in C^{1,\a}_{B}\subseteq C^{\a+1}_{Y}$ for any $1\leq i,j,l\leq p_0$%)
.}\end{proof}

\begin{proof}[Proof of Proposition \ref{prop:complementary_bis_alternative} for $n=\bar{n}$ and
$m=1$]
We assume {$u\in C^{2{\bar{n}}+1,\alpha}_{B}$} and we prove that, for any $0\leq k \leq \bar{n}$,
\begin{equation}\label{eq:ste14}
u(t,x+\xi)=T_{2\bar{n}+1} u\big( (t,x),(t,x+\xi) \big) + R_{\bar{n}}\big( t,x,\xi \big),
\end{equation}
with
\begin{equation}\label{eq:ste16}
|R_{\bar{n}}\big(t, x,\xi \big)| \le c_B \|u\|_{C^{2\bar{n}+1,\a}_{B}} |\x|_{B}^{2n+1+\a},\qquad (t,x)\in\R\times\Rd,\quad\xi\in \bigoplus_{j=0}^{k-1}V_{j}.%,\qquad t\in\R,\quad x,\x\in\R^{d},
\end{equation}
We prove the thesis by induction on $k$. For $k=0$ there is nothing to prove since $R_{\bar{n}}\big(t, x,0 \big)\equiv 0$. Now, assume $0\leq k <\bar{n}$,  {$\xi \in \bigoplus_{j=0}^{k-1} V_{j}$ and $ v\in V_{k}$}. Then
\begin{equation}
  %u(t,x+\xi+v)-T_{2{\bar{n}}+1} u\big( (t,x),(t,x+\xi+v) \big)
  R_{\bar{n}}(t,x,\xi+v)=F_1 +F_2,
\end{equation}
with
\begin{align}
F_1=& \,u(t,x+\xi+v)-T_{2{\bar{n}}+1} u\big( (t,x+v),(t,x+\xi+v) \big)\\
F_2=&\; T_{2{\bar{n}}+1} u\big( (t,x+v),(t,x+\xi+v)\big)-T_{2{\bar{n}}+1} u\big( (t,x),(t,x+\xi+v) \big).
\end{align}
 We can apply the inductive hypothesis on $F_1$ and obtain the estimate%s
 \begin{align}
  |F_1|&\leq c_B\|u\|_{C^{2\bar{n}+1,\a}_{B}}|\xi|_B^{2{\bar{n}}+1+\a}\leq c_B\|u\|_{C^{2\bar{n}+1,\a}_{B}}|\xi+v|_B^{2{\bar{n}}+1+\a}% && \text{if } u \in C^{2\bar{n}+1,\a}_{B}.
  .
 \end{align}
Recalling \eqref{eq:def_Tayolor_n}, $F_2$ can be written as
\begin{align}
 F_2 &={\sum_{0\leq |\b|_B\leq 2{\bar{n}}+1 \atop \b^{[i]}=0 \text{ if } i\geq k}} \frac{1}{\b !}  \p_{x}^{\b} u(t,x+v)\, \xi^{\b} - {\sum_{0\leq |\b|_B\leq 2{\bar{n}}+1 \atop \b^{[i]}=0 \text{ if } i\geq k} \sum_{0\leq |\gamma|_B\leq 2{\bar{n}}+1-|\beta|_B \atop \g=\gamma^{[k]}}} % \frac{1}{\b !} \sum_{|\g^{[k+1]}|_B=0}^{2{\bar{n}}+1-|\b|_B} \frac{1}{\g^{[k+1]}!}
 \frac{1}{\b !\g!} \, \p_{x}^{\g}\p_{x}^{\b} u(t,x) \, \xi^{\b}v^{\g}  \\
        &={\sum_{0\leq |\b|_B\leq 2{\bar{n}}+1 \atop \b^{[i]}=0 \text{ if } i\geq k}} \frac{1}{\b !} \bigg( \p_{x}^{\b} u(t,x+v)-  { \sum_{0\leq |\gamma|_B\leq 2{\bar{n}}+1-|\beta|_B \atop \g=\gamma^{[k]}}} % \frac{1}{\b !} \sum_{|\g^{[k+1]}|_B=0}^{2{\bar{n}}+1-|\b|_B} \frac{1}{\g^{[k+1]}!}
 \frac{1}{\g!} \, \p_{x}^{\g}\p_{x}^{\b} u(t,x) \, v^{\g} \bigg) \xi^{\b}  \\
            &= {\sum_{0\leq |\b|_B\leq 2{\bar{n}}+1 \atop \b^{[i]}=0 \text{ if } i\geq k}} \frac{1}{\b !} \Big( \p_{x}^{\b} u(t,x+v)- T_{2{\bar{n}}+1-|\b|_B}\p_x^{\b}u\big((t,x),(t,x+v)\big)\Big) \xi^{\b}.\label{eq:stefano20}
\end{align}
By Remark \ref{rem:step2}, we get %$\partial_{x}^{\beta}u \in C^{2{\bar{n}}+1 -|\beta|_B,\alpha}_{B,\text{\rm loc}}$ (
$\partial_{x}^{\beta}u \in C^{2{\bar{n}}+1 -|\beta|_B,\alpha}_{B}$%)
. Now, if $|\b|_B\geq 1$, we can apply %the inductive hypothesis of
Part 3 of Theorem \ref{th:main} for $n=2{\bar{n}}+1-|\b|_B$ on $\p_{x}^{\b} u$ and get
\begin{align}\label{eq:michele3}
  \left|\p_{x}^{\b} u(t,x+v)- T_{2{\bar{n}}+1-|\b|_B}\p_x^{\b}u((t,x),(t,x+v))\right| \left|\xi^{\b}\right|&\leq %\frac{1}{\b !}
c_B\|u\|_{C^{2\bar{n}+1,\a}_{B}}|v|_B^{2{\bar{n}}+1-|\b|_B+\a} |\xi|_B^{|\b|_B}\\
        &\leq
         c_B\|u\|_{C^{2\bar{n}+1,\a}_{B}}|\xi+v|_B^{2{\bar{n}}+1+\a}.
\end{align}
On the other hand, if $|\b|_B =0$ then we have to estimate
\begin{equation}
\label{eq:michele1}
 u(t,x+v)-{\sum_{0\leq |\gamma|_B\leq 2{\bar{n}}+1 \atop \g=\gamma^{[k]}}}  \frac{1}{\g!} \, \p_{x}^{\g} u(t,x) \, v^{\g}%-\sum_{|\g^{[k+1]}|_B\leq 2{\bar{n}}+1} \frac{1}{\g^{[k+1]} !}  \p_{x}^{\g^{[k+1]}} u(t,x) v^{\g^{[k+1]}}
.
\end{equation}
Recall that, by definition, we have %$|\g^{[k+1]}|_B =(2(k+1)+1)|\g^{[k+1]}|$
{$|\g|_B =(2 k+1)|\g|$} if $\gamma=\gamma^{[k]}$. Now, set
\begin{equation}\label{eq:ste21}
 {j:=\max \{i\geq 0 \mid (2k+1)i\leq 2{\bar{n}}+1\}},
\end{equation}
and note that $j\geq 1$ because $k < \bar{n}$. By Remark \ref{rem:step2} and the mean-value
theorem, there exists a family of vectors $(\bar{v}_{\eta})_{\eta\in \mathcal{I}_{k}^j}$ where
\begin{equation}\label{eq:ste22}
 {\mathcal{I}_{k}^j=\{\eta\in\N_0^d \mid \eta=\eta^{[k]} \text{ and }|\eta|_{B}=(2k+1)j \},}
\end{equation}
%$(\bar{v}_i)_{i\in \mathcal{I}_{k}^j}$, where $\mathcal{I}_{k}^j=\{\bar{p}_{k-1}+1,\dots,\bar{p}_{k}\}^j$,
such that $\bar{v}_{\eta}\in V_k$, $|\bar{v}_{\eta}|\leq |v|$ and
\begin{equation}
  u(t,x+v)-{\sum_{0\leq |\gamma|_B\leq (2k+1)(j-1) \atop \gamma=\gamma^{[i]}}}  \frac{v^{\g}}{\g!} \p_{x}^{\g} u(t,x)
  = \sum_{\eta \in \mathcal{I}_{k}^j} \frac{v^{\eta}}{\eta!}  \p_{x}^{\eta} u(t,x+\bar{v}_{\eta}). %\frac{1}{j!} \sum_{I\in \mathcal{I}_{k+1}^j}  \p_{x}^{I} u(t,x+\bar{v}_I) v^{I}.
\end{equation}
%With this identity equation \eqref{eq:michele1} reads as
Therefore, we obtain
\begin{align}
 &\Big| u(t,x+v)-{\sum_{0\leq |\gamma|_B\leq 2{\bar{n}}+1 \atop \g=\gamma^{[i]}}}  \frac{v^{\g}}{\g!} \p_{x}^{\g} u(t,x)\Big|
 =\Big|
  \sum_{\eta \in \mathcal{I}_{k}^j} \frac{v^{\eta}}{\eta!}\left( \p_{x}^{\eta} u(t,x+\bar{v}_{\eta}) -\p_{x}^{\eta} u(t,x) \right)\Big|=
 \intertext{(by \eqref{eq:ste21})}
 &\hspace{70pt}=
  \Big|\sum_{\eta\in \mathcal{I}_{k}^j} \frac{1}{\eta!}
  \Big( \p_{x}^{\eta} u(t,x+\bar{v}_{\eta}) -T_{2{\bar{n}}+1-(2k+1)j}\p_{x}^{\eta} u\big((t,x),(t,x+\bar{v}_{\eta})\big) \Big)
  v^{\eta}\Big|\le
\intertext{(by Remark \ref{rem:step2}, $\p_{x}^{\eta} u \in C^{2{\bar{n}}+1-(2k+1)j,\a}_{B}$ and
 thus by Part 3 of Theorem \ref{th:main} with $n=2{\bar{n}}+1-(2k+1)j$)}
 &\hspace{70pt} \leq c_B\|u\|_{C^{2\bar{n}+1,\a}_{B}} \sum_{\eta\in \mathcal{I}_{k}^j} \frac{1}{\eta!}\,
|\bar{v}_{\eta}|_B^{2{\bar{n}}+1-(2k+1)j+\alpha}\, |v|_B^{|\eta|_B}\le
 \intertext{(since $|\bar{v}_{\eta}|\leq |v|$ and by \eqref{eq:ste22})}
 &\hspace{70pt}\leq c_B\|u\|_{C^{2\bar{n}+1,\a}_{B}}|v|_B^{2{\bar{n}}+1+\a}\leq c_B\|u\|_{C^{2\bar{n}+1,\a}_{B}}|\xi+v|_B^{2{\bar{n}}+1+\a}
\end{align}
which concludes the proof.\end{proof}

\subsubsection{Proof of Proposition \ref{prop:complementary_new} for $\bf{n={\bar{n}}}$ }

To start %fix $z=(t,x)$, $\z=(t,\xi)\in\R\times\R^d$ (no increment in the time-variable), and
we show that if {$u\in C^{2{\bar{n}}+1,\a}_{B}$} then for any $z=(t,x)$,
$\z=(t,\xi)\in\R\times\R^d$ we have
\begin{equation}\label{eq:Tay_disp_spaz}
 \left|u(t,x) - \bar{T}_{2{\bar{n}}+1} u\left((t,\xi),(t,x)\right)\right|
 \leq c_B\|u\|_{C^{2\bar{n}+1,\a}_{B}}\, |x-\xi|_B^{2{\bar{n}}+1+\a}% \quad \text{as } |x-\xi|_B \to 0
.
\end{equation}
Define the point $\bar{z}=(t,\bar{x})$ with
\begin{equation}
   \xb^{[i]}=
    \begin{cases}
    x^{[i]}, & \text{ if $i\geq {\bar{n}}$},\\
    \xi^{[i]}, & \text{ if $i< {\bar{n}}$}.
    \end{cases}
\end{equation}
It follows that
\begin{equation}\label{def:zbar}
   (x-\xb)^{\b}=
    \begin{cases}
    (x-\xi)^{\b} & \text{ if $|\b|_B\leq 2{\bar{n}}+1$, $\b^{[{\bar{n}}]}=0$},\\
    0, & \text{ if $|\b|_B\leq 2{\bar{n}}+1$, $\b^{[{\bar{n}}]}\neq 0$},
    \end{cases}
\end{equation}
and%, in particular,
\begin{equation}
 |x-\xb|_B\leq |x-\xi|_B, \qquad  |\bar{x}-\xi|_B \leq  |x-\xi|_B.
\end{equation}
Then we write
\begin{align}
 u(t,x) - \bar{T}_{2{\bar{n}}+1} u\left((t,\xi),(t,x)\right)= F_1 +F_2,
\end{align}
with
\begin{equation}
 F_1 = u(t,x) - \bar{T}_{2{\bar{n}}+1} u\left((t,\xb),(t,x)\right),\qquad %\\
%&
 F_2= \bar{T}_{2{\bar{n}}+1} u\left((t,\xb),(t,x)\right)-\bar{T}_{2{\bar{n}}+1} u\left((t,\xi),(t,x)\right).
\end{equation}
Applying Proposition \ref{prop:complementary_bis_alternative} with $n=\bar{n}$ and $m=1$, %on the first difference
we obtain %the bound
\begin{equation}
|F_1|\leq c_B\|u\|_{C^{2\bar{n}+1,\a}_{B}} |x-\xb|_B^{2{\bar{n}}+1+\a}\leq c_B\|u\|_{C^{2\bar{n}+1,\a}_{B}} |x-\xi|_B^{2{\bar{n}}+1+\a}%,\quad \text{ as } |x-\xi|_B \to 0
.
\end{equation}
Now, by \eqref{def:zbar} we have
\begin{align}
 %\bar{T}_{2{\bar{n}}+1} u((t,\xb),(t,x))-\bar{T}_{2{\bar{n}}+1} u((t,\xi),(t,x))
 F_2=& \sum_{\genfrac{}{}{0pt}{}{|\b|_B\leq 2{\bar{n}}+1}{\b^{[{\bar{n}}]}=0}} \frac{1}{\b !}
\Big(\p_{x}^{\b} u(t,\xb) - \p_{x}^{\b} u(t,\xi)\Big) (x-\xi)^{\b}%\\
                        %&
                        -{\sum_{i=\bar{p}_{{\bar{n}}-1}+1}^{\bar{p}_{\bar{n}}}} Y^{({\bar{n}})}_{v_i^{({\bar{n}})}} u(t,\xi) {(x-\xi)_{i}}.
\end{align}
Moreover, by Remark \ref{rem:step2} we have {$\p_{x}^{\b} u\in C^{2\bar{n}+1-|\b|_B,\a}_{B}$} and
therefore, if $|\b|_B>0$, by {Part 3} of  Theorem \ref{th:main} for $n=2\bar{n}+1-|\b|_B$, we get
%Each term of the first sum with $|\b|_B>0$ is
{$$
\Big|\Big(\p_{x}^{\b} u(t,\xb) - \p_{x}^{\b} u(t,\xi)\Big) (x-\xi)^{\b}\Big|\leq c_B\|u\|_{C^{2\bar{n}+1,\a}_{B}} |\xb-\xi|_B^{2{\bar{n}}+1+\a - |\b|_B }|x-\xi|^{|\b|}\leq c_B\|u\|_{C^{2\bar{n}+1,\a}_{B}} |x-\xi|_B^{2{\bar{n}}+1+\a }%\big),\quad \text{as } |x-\xi|_B \to 0
.
$$}% by the inductive hypothesis on Theorem \ref{th:main} since by Remark \ref{rem:step2} $\p_{x}^{\b} u\in C^{2\bar{n}+1-|\b|_B,\a}_{B,\text{\rm loc}}$.
In order to conclude the proof of \eqref{eq:Tay_disp_spaz}, we only have to prove
{\begin{equation}\label{eq:ste30}
 \Big| u(t,\xb) -u(t,\xi)- \sum_{j=\bar{p}_{{\bar{n}}-1}+1}^{\bar{p}_{\bar{n}}} Y^{({\bar{n}})}_{v_j^{({\bar{n}})}} u(t,\xi) (x-\xi)_{j}\Big|\leq c_B\|u\|_{C^{2\bar{n}+1,\a}_{B}} |x-\xi|_B^{2{\bar{n}}+1+\a }%\big),\quad \text{as } |x-\xi|_B \to 0
.
\end{equation}
}
We set the points $\zeta_{i}=(t,\xi_{i})$, for $i=\bar{n}-1,\cdots,r$, as defined in Lemma
\ref{lem:connecting_curves} for $n=\bar{n}$ and $v=\bar{x}-\xi$. By \eqref{eq:ste25} we have
\begin{align}
 \bar{T}_{2{\bar{n}}+1} u(\z_{i-1},\z_i)= u(\z_{i-1}),\qquad i={\bar{n}},\dots,r,
\intertext{and}
 |\d_i|\leq c_B |\bar{x}-\xi|_B \leq c_B |x-\xi|_B,\qquad i={\bar{n}},\dots,r. \label{eq:stima_delta}
\end{align}
It is now clear that
\begin{equation}
u(t,\xb) -u(t,\xi)- \sum_{j=\bar{p}_{{\bar{n}}-1}+1}^{\bar{p}_{\bar{n}}} Y^{({\bar{n}})}_{v_j^{({\bar{n}})}} u(t,\xi) (x-\xi)_{j} =u(\z_r)-\bar{T}_{2{\bar{n}}+1} u(\z_{{\bar{n}}-1},\z_{\bar{n}})%\\
%&
= \sum_{i={\bar{n}}}^r \Big( u(\z_i)-\bar{T}_{2{\bar{n}}+1} u(\z_{i-1},\z_i) \Big),
\end{equation}
and formula \eqref{eq:ste30} follows from Proposition \ref{prop:complementary_alternative}  along with \eqref{eq:stima_delta}%{\blue (Perche' serve \eqref{normstep}?)}
.

We are now ready to prove \eqref{eq:drivative_commutators} for $n=\bar{n}$. For any $i\in
\{\bar{p}_{{\bar{n}}-1}+1,\dots, \bar{p}_{\bar{n}}\}$ and $\d\in \R$, set $x=\xi+\d e_{i}$ in
\eqref{eq:Tay_disp_spaz}, where $e_{i}$ is the $i$-th vector of the canonical basis of $\R^d$: we
obtain
\begin{equation}\label{eq:esistenza_der}
u(t,\xi+\d e_{i}) -u(t,\xi) - \d Y^{({\bar{n}})}_{v_i^{({\bar{n}})}}u(t,\xi)
=O\big(|\d|^{1+\frac{\a}{2{\bar{n}}+1}}\big),\quad \text{ as } \delta \to 0.
\end{equation}
This implies that $\p_{x_{i}}u(t,\xi)$ exists and
\begin{equation}\label{eq:esistenza_der_2}
\p_{x_{i}}u(t,\xi)=Y^{({\bar{n}})}_{v_i^{({\bar{n}})}}u(t,\xi) \quad t\in \R,\:\: \xi \in \R^d, \:\: i=\bar{p}_{{\bar{n}}-1}+1,\dots, \bar{p}_{\bar{n}}.
\end{equation}
Finally, by Remark \ref{rem:step2} we have $%Y^{({\bar{n}})}_{v_i^{({\bar{n}})}}u \in C^{0,\a}_{B,\text{\rm loc}}$ $\big(
Y^{({\bar{n}})}_{v_i^{({\bar{n}})}}u \in C^{0,\a}_{B}%\big)
$ and thus $%\p_{x_{i}}u \in C^{0,\a}_{B,\text{\rm loc}}$ $\big(
\p_{x_{i}}u \in C^{0,\a}_{B}%\big)
$.

{\begin{remark}\label{rem:michele} Incidentally we have just proved a special case of %Part 2 (
Part 3 %)
of Theorem \ref{th:main} for $n=2\bar{n}+1$, namely the case when there is no increment
in the time variable. Precisely we have shown that, for any function $u\in
C^{2\bar{n}+1}_{B}$, we have
\begin{equation}\label{eq:tay_solo_spazio_glob}
  \big| u(t,x)-T_{2\bar{n}+1} u\big((t,\xi),(t,x)\big) \big|\le c_B \|u\|_{C^{2\bar{n}+1,\a}_{B}}  {|x-\xi |_{B}^{2\bar{n}+1+\alpha}}, \qquad t\in\R,\quad x,\xi\in \R^d.
\end{equation}
\end{remark}
}

\subsubsection{Proof of Part 3 of Theorem \ref{th:main} for $n=2{\bar{n}}+1$}\label{subsec:prova_th_dispari}

{Relation \eqref{eq:ste31}} is a trivial consequence of Remark \ref{rem:ste1} {(see
\eqref{eq:ste32}-\eqref{eq:ste4})} along with Proposition \ref{prop:complementary_new} for
$n=\bar{n}$. We next prove estimate \eqref{eq:estim_tay_n}: by \eqref{eq:ste31}, for any $z=(t,x)$
and $\z=(s,\xi)$, the B-Taylor polynomial $T_{2{\bar{n}}+1}u(\z, z)$ is well defined. Define the
point $\z_1:=e^{(t-s)Y}(\z)=(t,e^{(t-s)B}\xi)$ and note that $\z_1$ and $z$ only differ {in} the
spatial variables. Moreover, we have
\begin{equation}
 \z_{1}^{-1}\circ z = \left(0,x-e^{(t-s)B}\xi\right),\qquad \z^{-1}\circ z =\left(t-s,x-e^{(t-s)B}\xi\right),
\end{equation}
and therefore
\begin{equation}\label{eq:norme_mic}
\norm{\z_{1}^{-1}\circ z}_B = \big|x-e^{(t-s)B}\xi\big|_B \leq \norm{\z^{-1}\circ z}_B.
\end{equation}
Now write
\begin{equation}
u(z) - T_{2{\bar{n}}+1} u(\z,z) = F_1 + F_2,
\end{equation}
with
\begin{equation}
F_1=u(z) - T_{2{\bar{n}}+1} u(\z_1,z), \qquad  F_2=T_{2{\bar{n}}+1} u(\z_1,z) -T_{2{\bar{n}}+1} u(\z,z).
\end{equation}
By \eqref{eq:tay_solo_spazio_glob} %(\eqref{eq:tay_solo_spazio_glob})
in Remark \ref{rem:michele} along with \eqref{eq:norme_mic}, we obtain the estimate
 $$
  |F_1|\leq c_B \|u\|_{C^{2\bar{n}+1,\a}_{B}} \|\z^{-1}\circ z\|_B^{2{\bar{n}}+1+\a}%),\quad \text{as }\|\z^{-1}\circ z\|_B \to 0
.$$
 %($c_B{C(2\bar{n}+1)}\norm{u}_{C^{2{\bar{n}}+1,\a}_B}\|\z^{-1}\circ z\|_B^{2{\bar{n}}+1+\a}$).
A convenient rearrangement of the terms in the Taylor polynomials  allows us to estimate $F_2$.
Precisely, we have
\begin{align}
%T_{2{\bar{n}}+1} u(\z_1,z) -T_{2{\bar{n}}+1} u(\z,z)
  F_2=&  \sum_{|\b|_B\leq 2{\bar{n}}+1}\frac{1}{\b !} \big(\partial_{\xi}^{\beta} u(e^{(t-s)Y}(\z))\big) (x-e^{(t-s)B}\xi)^{\b} %\\
                                                                        %-&
      - \sum_{2k + |\b|_B\leq 2{\bar{n}}+1}\frac{Y^k \partial_{\xi}^{\beta}u(\z)}{\b !k!}   (x-e^{(t-s)B}\xi)^{\b}(t-s)^k\\
    =& \sum_{|\b|_B\leq 2{\bar{n}}+1}\frac{1}{\b !} \left(\partial_{\xi}^{\beta} u(e^{(t-s)Y}(\z))
     - \sum_{2k\leq 2{\bar{n}}+1-|\b|_B}%^{\left \lfloor{\frac{2{\bar{n}}+1-|\b|_B}{2}}\right \rfloor}
     \frac{(t-s)^k}{k!} Y^k\p_{\xi}^{\b} u(\z)\right)(x-e^{(t-s)B}\xi)^{\b}.
\end{align}
{Now, by \eqref{eq:ste31} we have $\p_{x}^{\b} u \in C^{2{\bar{n}}+1-|\b|_B,\a}_{B}$ and thus, by
Remark \ref{rem:eq:mean_value_Yn} we obtain
\begin{equation}
|F_2| \leq %c_B
\|u\|_{C^{2\bar{n}+1,\a}_{B}} \sum_{|\b|_B\leq 2{\bar{n}}+1}\frac{1}{\b !} |t-s|^{\frac{2{\bar{n}}+1-|\b|_B+\a}{2}}\, \big|x-e^{(t-s)B}\xi\big|_B^{|\b|_B}\leq c_B \|u\|_{C^{2\bar{n}+1,\a}_{B}} \|\z^{-1}\circ z\|_B^{2{\bar{n}}+1+\a},
\end{equation}
}
and this concludes the proof.

\subsection{Step 3}
Fix $\bar{n}\in \{0,\cdots,r-1 \}$. Assume to be holding true:
\begin{enumerate}
\item[-] Proposition \ref{prop:complementary_new} %to hold true
for any $0\leq n\leq \bar{n}$;
\item[-] Theorem \ref{th:main} %and Propositions \ref{prop:complementary}, \ref{prop:complementary_bis}
%to hold true
for any $0\leq n\leq 2\bar{n}+1$;
\end{enumerate}
we have to prove:
\begin{enumerate}
\item[-] Propositions \ref{prop:complementary_alternative} and \ref{prop:complementary_bis_alternative} for $n=\bar{n}+1$, $m=0$;
%\item[-] Proposition \ref{prop:complementary_new} for $n=\bar{n}$;
\item[-] Part 3 of Theorem \ref{th:main} for $n=2\bar{n}+2$.
\end{enumerate}

In this case, the proof is relatively simpler if compared to the one of Step 2. This is because we
do not need to prove the existence of the Euclidean derivatives of the higher level. Hence the
proofs are simpler versions of those in Step 2. We skip the details for the sake of brevity.

\subsection{Step 4}
Here we fix a certain $\bar{n}\geq 2r+1$, suppose Theorem \ref{th:main} true for any $0\le n\le
\bar{n}$ and prove {Part 3 of Theorem \ref{th:main}} for $n=\bar{n}+1$. To prove the claim, we
will first consider the case with no increment w.r.t. the time variable, as we have done in Step
2. In that case, we used the curves $\gamma^{n,k}_{v,\delta}(z)$ in order to increment those
variables w.r.t. which we had no regularity in the Euclidean sense: then we applied Proposition
\ref{prop:complementary_alternative} to estimate  the increment along such curves.
This time, this will not be necessary because, since $\bar{n}+1> 2r+1$, the existence of the Euclidean derivatives is ensured along any direction by the inductive hypothesis. 
\begin{proof}[Proof of Part 3 of Theorem \ref{th:main} for $n=\bar{n}+1$]
Recall that, by hypothesis, {$u\in C^{\bar{n}+1,\a}_{B}$} with $\bar{n}\geq 2r+1$. It is easy to
prove that,
for any $z=(t,x)$, $\z=(s,\xi) \in \R^d$,
%. In particular, if $u\in C^{\bar{n}+1,\a}_{B}$, then
we %also
have
\begin{equation}\label{eq:bound_glob_space}
\left|u(t,x)-T_{\bar{n}+1}u((t,\xi),(t,x))\right|\leq c_B \norm{u}_{C^{\bar{n}+1,\a}_{B}}|x-\xi|_{B}^{\bar{n}+1+\a}.
\end{equation}
%if $u\in C^{\bar{n}+1,\a}_{B}$,
The proof of the latter identity is identical to that of Proposition
\ref{prop:complementary_bis_alternative}. Precisely, under the assumption $\bar{n}\geq 2r+1 $, the
technical restriction made on the spatial increments in Proposition
\ref{prop:complementary_bis_alternative} can be dropped and the proof proceeds exactly in the same
way, {by making sure that the constant $c_B$ in \eqref{eq:bound_glob_space} is actually
independent of $\bar{n}$}.

The proof of {Part 3 of} Theorem \ref{th:main} then follows exactly as in Step 2, by using the estimate %\eqref{eq:bound_loc_space}-
\eqref{eq:bound_glob_space} instead of %\eqref{eq:tay_solo_spazio_loc}-
\eqref{eq:tay_solo_spazio_glob}.
\end{proof}
\subsection{Proof of the local version of Theorem \ref{th:main} (Part 1 and Part 2)}

The proof of the local version of Theorem \ref{th:main} is based upon the same arguments used to prove its global counterpart. The main additional difficulty arising when proving Part 1 and Part 2, is to make sure that all the integral curves used in the proof to connect $z$ to $\z$ do not exit the domain $\Omega$. There comes the necessity to take $z$ in a small ball centered at $\z$ with radius $r$. In particular, we have to check that all the connecting curves are contained in a bigger ball with radius $R>r$, compactly contained in $\Omega$, in order to bound the remainder in \eqref{eq:estim_tay_n_loc} by means of the $C^{n,\alpha}_{B,\text{loc}}$ norm of $u$ on such ball.
   %extremely

To synthesize, 
the proof could be summarized as follows. First prove the Taylor estimate \eqref{eq:estim_tay_n_loc} for two points $z,\xi\in \Rdd$ with the same time-component (estimate \eqref{eq:tay_solo_spazio_glob}); this also proves the existence of those Euclidian derivatives whose existence is not directly implied by definition of $C^{n,\alpha}_{B,\text{loc}}$% and by the inductive hypothesis
and thus proves Part 1 of the theorem. To do this, one can proceed as in Section
\ref{prop:complementary_new}: precisely, one would first apply the local version of Proposition
\ref{prop:complementary_bis_alternative}, whose proof is exactly analogous to its global
counterpart, to control the increment of $u$ between $\xi$ and $\bar{x}$. Secondly, one would
define the points $\zeta_{k}=(t,\xi_{k})$ for $k=\bar{n}-1,\cdots,r$, by means of the {curves} of
Lemma \ref{prop:complementary_new} and thank to  \eqref{eq:ste25}-\eqref{eq:ste34}, obtain the
bound
\begin{equation}
|\d_k|+\norm{\z_{k}^{-1}\circ \z}_B \leq c_B |x-\xi|_B \leq c_B r, \qquad k={\bar{n}},\dots,r.
\end{equation}
This would ensure that {each point $\z_{k}$} is inside the domain $\Omega$, for any $r$ suitably
small, and would allow to  control the increment of $u$ between $\z_{k-1}$ and $\z_{k}$ by means
of the local version of Proposition \ref{prop:complementary_alternative}. The latter preliminary
result can be proved analogously to the global case, by making use of the bounds \eqref{normstep}
and \eqref{eq:ste12} to control the distance of each curve $\gamma^{n,k}_{v,\delta}(z)$ from $\z$.

Eventually, the proof of Part 2 for two general points $z,\z\in\Rdd$ follows by moving along the
integral curve of $Y$ to control the increment in the time-variable and by using the bound
\eqref{eq:ste12} to check that $e^{(t-s)Y}(\z)\in\Omega$.

%%%%%%%%%%%%%%%%%%%%%%%%%%%%%%%%%%%%%%%%%%%%%%%%%%%%
%
%           Bibliography
%
%%%%%%%%%%%%%%%%%%%%%%%%%%%%%%%%%%%%%%%%%%%%%%%%%%%%

\bibliographystyle{chicago}
\bibliography{Bibtex-Master-3.00}

\end{document}